\documentclass[
a4paper
,11pt
]{amsart}

\usepackage[T1]{fontenc}
\usepackage[utf8]{inputenc}
\usepackage[inner=3cm,outer=3cm]{geometry}

\usepackage{hyperref}
\hypersetup{
   backref
   ,colorlinks=true
   ,allcolors=black
}

\usepackage{amssymb}
\usepackage[foot]{amsaddr}
\usepackage{enumerate}
\usepackage{xcolor}
\usepackage{mathtools}
\usepackage{booktabs}
\usepackage{tikz}
\usetikzlibrary{matrix,arrows}
\usepackage{thmtools}
\declaretheorem[parent=section]{theorem}
\declaretheorem[sibling=theorem]{proposition}
\declaretheorem[sibling=theorem]{lemma}
\declaretheorem[sibling=theorem]{corollary}

\declaretheorem[sibling=theorem,style=remark]{remark}

\declaretheorem[sibling=theorem,style=definition]{definition}
\declaretheorem[numbered=no,name=Theorem]{theorem*}
\declaretheorem[numbered=no,style=definition,name=Definition]{definition*}
\declaretheorem[numbered=no,name=Proposition]{proposition*}
\IfFileExists{todonotes.sty}{%
  \usepackage{todonotes}%
  }{%
}
\newcommand{\ARXIV}[1]{\href{https://arXiv.org/abs/#1}{arXiv:#1}}

\usepackage{mathrsfs}
\begin{document}

\bibliographystyle{amsalpha}



\renewcommand{\d}[1]{\,\ensuremath{\mathrm{d}} #1}
\renewcommand{\L}{\operatorname{L}}
\renewcommand{\P}{\operatorname{P}}
\newcommand{\D}{\operatorname{D}}
\newcommand{\E}{\operatorname{E}}
\newcommand{\Id}{\operatorname{Id}}
\newcommand{\I}{\operatorname{I}}
\newcommand{\var}{\operatorname{Var}}
\newcommand{\cov}{\operatorname{Cov}}
\newcommand{\mi}{\ensuremath{\mathrm{i}}}
\newcommand{\me}{\ensuremath{\mathrm{e}}}
\newcommand{\N}{\mathbb{N}}
\newcommand{\Z}{\mathbb{Z}}
\newcommand{\Q}{\mathbb{Q}}
\newcommand{\R}{\mathbb{R}}
\newcommand{\C}{\mathbb{C}}
\newcommand{\on}[1]{\operatorname{#1}}
\newcommand{\tr}{\on{tr}}
\newcommand{\norm}[1]{\left\lVert #1 \right\rVert}
\newcommand{\Ex}[1]{\operatorname{E}\! \left[ #1 \right]}
\newcommand{\abs}[1]{\left\vert #1 \right\rvert}
\newcommand{\Ker}[1]{\operatorname{ker}\left(#1\right)}

\author{Solesne Bourguin$^1$}
\address{$^1$Boston University, Department of Mathematics and Statistics,  111 Cummington Mall, Boston, MA 02215, USA}
\email{bourguin@math.bu.edu}
\author{Simon Campese$^2$}
\address{$^2$University of Luxembourg, Mathematics Research Unit,
 6 Avenue de la Fonte, 4364 Esch-sur-Alzette, Luxembourg}
\email{simon.campese@uni.lu}
\title[Approximation of Hilbert-valued Gaussians on Dirichlet structures]{Approximation of Hilbert-valued Gaussians on Dirichlet structures}
\thanks{S. Bourguin was supported in part by the Simons Foundation
  grant 635136}
\begin{abstract}
  We introduce a framework to derive quantitative central limit
  theorems in the context of non-linear approximation of Gaussian random variables taking values in a separable Hilbert space. In particular, our method provides an alternative to the usual (non-quantitative) finite dimensional distribution convergence and tightness argument for proving functional convergence of stochastic processes.
  We also derive four moments bounds for Hilbert-valued random variables with possibly infinite chaos expansion, which include, as special cases, all finite-dimensional four moments results for Gaussian approximation in a  diffusive context proved earlier by various authors.
  Our main ingredient is a combination of an infinite-dimensional version of Stein's method as developed by Shih and the so-called Gamma calculus.
  As an application, rates of convergence for the functional Breuer-Major theorem are established.
\end{abstract}
\maketitle

\section{Introduction}

Random variables taking values in Hilbert spaces play an important
role in many fields of mathematics and statistics, both at a
theoretical and applied level. For example, they arise naturally in
statistics, in particular in the field of functional data analysis or
machine learning (for example in the context of Reproducing Kernel Hilbert Spaces). An important and classical topic is the asymptotic analysis of sequences of such random variables.

In the linear case, i.e., when looking at normalized sums of i.i.d.\ random variables, the asymptotic behaviour is very well understood, with central limit theorems including error bounds being available in Banach or more general infinite-dimensional spaces (see~\cite{araujo-gine:1980:central-limit-theorem}). Here, (separable) Hilbert spaces have the distinguished property of being the only infinite-dimensional Banach spaces for which convergence of such sums is equivalent to finite variances (square integrability of the norms) of the components.

In the non-linear case, where the sum is replaced by a general transformation, much less is known, except when the dimension of the Hilbert space is finite. In this case, Nourdin and Peccati (\cite{nourdin-peccati:2009:steins-method-wiener}) have introduced the very powerful combination of Stein's method and Malliavin calculus, which yields quantitative central limit theorems for a very wide class of square integrable real-valued transformations of arbitrary Gaussian processes. Since its inception, this approach, which is now known as the \emph{Malliavin-Stein method}, has had a very substantial impact with numerous generalizations and applications. We refer to the  monograph~\cite{nourdin-peccati:2012:normal-approximations-malliavin} for an overview.

In this paper, we lift the theory to infinite-dimension, thus obtaining quantitative central limit theorems for square-integrable and Hilbert-valued random variables. The setting we will be working in is that of a diffusive Markov generator $L$, acting on $L^2(\Omega;K)$, where $K$ is a real separable Hilbert space. Our main result (see Section~\ref{sec:appr-hilb-valu} for unexplained definitions and Theorem~\ref{thm:1} for a precise statement) then states that for random variables $F$ in the domain of the associated carré du champ operator $\Gamma$ and centered, non-degenerate Gaussians $Z$ on $K$ with covariance operator $S$, one has
\begin{equation}
 \label{eq:56}
 d(F,Z) \leq \frac{1}{2} \sqrt{
\E \left(
   \left\lVert \Gamma(F,-L^{-1}F) - S \right\rVert_{\on{HS}}^2
 \right)}.
\end{equation}
Here, $\left\lVert \cdot \right\rVert_{\on{HS}}$ denotes the Hilbert-Schmidt norm, $L^{-1}$ the pseudo-inverse of the generator $L$ and $d$ is a probability metric generating a topology which is stronger than convergence in distribution.

Some examples of random variables $F$ fitting our framework are
homogeneous sums of i.i.d.\ Gaussians with Hilbert-valued coefficients
(or more generally a polynomial chaos with distributions coming from a diffusion generator), stochastic integrals of the form $F_t = \int_0^{\infty} u_{t,s} \d{B_s}$, where $B$ is Brownian motion and the kernel $u$ is such that the trajectories of $F$ are Hölder-continuous of order less than one half, or multiple Wiener-Itô integrals.

Proceeding from the general bound~\eqref{eq:56}, we generalize and refine the two most important results of the finite-dimensional Malliavin-Stein framework:

The first results are quantifications of so-called~\emph{Fourth Moment Theorems} (first discovered in~\cite{nualart-peccati:2005:central-limit-theorems} and substantially generalized in~\cite{ledoux:2012:chaos-markov-operator%
  ,azmoodeh-campese-poly:2014:fourth-moment-theorems%
  ,bourguin-campese-leonenko-ea:2019:four-moments-theorems}), which
state that for a sequence of eigenfunctions of the carré du champ
operator satisfying a chaotic property, convergence in distribution to
a Gaussian is equivalent to convergence of the second and fourth
moment. We prove that such quantitative Fourth Moment Theorems
continue to hold in infinite-dimension, i.e., that if $F$ is a chaotic
eigenfunction of the carré du champ operator and $Z$ is a Gaussian having the same covariance operator as $F$, then one has (see Section~\ref{sec:fourth-moment-bounds} for precise statements)
\begin{equation}
 \label{eq:48}
 d(F,Z) \leq \frac{1}{2}
 \sqrt{
   \E \left( \left\lVert F \right\rVert^4 \right)
   \left(
     \E \left( \left\lVert F \right\rVert^4 \right)
     -
     \E \left( \left\lVert Z \right\rVert^4 \right)
   \right)
 }.
\end{equation}
The fact that the moment difference on the right-hand side is non-negative will follow from our analysis. In fact, we prove a more general version of~\eqref{eq:48} for $K$-valued random variables $F$ which have a possibly infinite chaos expansion and whose covariance operator not necessarily coincides with the one of $Z$. Even in finite dimensions, such bounds are new.

The second type of results we obtain are so called \emph{contraction bounds} in the special case
where $L$ is the Ornstein-Uhlenbeck generator. Here, the chaotic
eigenfunctions of the generator become multiple Hilbert-valued
Wiener-It\^o integrals, so that, using a Hilbert-valued version of Malliavin calculus, their moments can be expressed in terms of contraction norms of their kernels (see~\eqref{eq:55} for a definition). Such contractions are important in applications due to their relatively straightforward computability. We provide bounds on $d(F,Z)$ ($F$ and $Z$ as in~\eqref{eq:56}) in terms of such contraction norms, again for random variables with possibly infinite chaos expansions. As an application, we explain how these bounds can be used to obtain rates of convergence in the functional Breuer-Major Theorem.

From a theoretical point of view, our results described above contain and extend, in a unified way, all previously established fourth moment theorems as well as carré du champ, fourth moment and contraction bounds for Gaussian approximation in a diffusive, finite-dimensional context (see
\cite{%
ledoux:2012:chaos-markov-operator%
,azmoodeh-campese-poly:2014:fourth-moment-theorems%
,campese-nourdin-peccati-ea:2016:multivariate-gaussian-approximations%
,bourguin-campese-leonenko-ea:2019:four-moments-theorems%
,nourdin-peccati:2009:steins-method-wiener%
,nourdin-peccati-reveillac:2010:multivariate-normal-approximation%
,noreddine-nourdin:2011:gaussian-approximation-vector-valued%
,nualart-ortiz-latorre:2008:central-limit-theorems%
,nualart-peccati:2005:central-limit-theorems%
}).

Furthermore, in the context of weak convergence of stochastic processes, our approach is an alternative to the usual method of proving finite-dimensional distribution convergence and tightness, with the advantage of yielding rates of convergence.

The existing literature on quantitative limit theorems in a non-linear and infinite-dimensional context is rather scarce. Barbour extended Stein's method to a functional setting in~\cite{barbour:1990:steins-method-diffusion} for diffusion approximations by a Brownian motion. This has recently been applied and extended by Kasprzak in~\cite{%
kasprzak:2017:steins-method-multivariate%
,kasprzak:2017:diffusion-approximations-via%
,kasprzak:2017:multivariate-functional-approximations%
}. Coutin and Decreusefond (\cite{coutin-decreusefond:2013:steins-method-brownian}) combined Stein's
method with integration by parts techniques in a separable Hilbert space
setting. While the general theme of this latter reference is similar to ours, the results are very different: the bounds in~\cite{coutin-decreusefond:2013:steins-method-brownian} are stated in terms of partial traces and require explicit evaluations of
isometries as all calculations are done in $\ell^2(\N)$; furthermore, no carr\'e du champ, moment, or contraction bounds are provided.

The rest of the paper is organized as follows. In Section~\ref{sec:preliminaries}, after recalling basic notions of probability theory on Hilbert spaces (Subsection~\ref{sec:prob-hilb-spac}), we provide an outline of Stein's method on abstract Wiener spaces (Subsection~\ref{sec:gauss-meas-steins}) and introduce the Dirichlet structure framework we will be working in (Subsection~\ref{sec:dirichl-struct-hilb}). The main results are contained in Section~\ref{sec:appr-hilb-valu}: we start by proving the aforementioned general carr\'e du champ bound~\eqref{eq:56} in Section~\ref{sec:an-abstract-carre} and then provide quantitative fourth moment theorems in Section~\ref{sec:fourth-moment-bounds}.
Section~\ref{sec:hilb-valu-wien} is devoted to the special case of the
Ornstein-Uhlenbeck generator. First, we present the associated
Malliavin calculus in a Hilbert space setting (Subsection~\ref{sec:malliavin-calculus}), which then leads to quantitative refined Fourth Moment and contraction bounds (Subsection~\ref{sec:fourth-moment-contr}). Finally, in Section~\ref{sec:quant-funct-breu}, rates of convergence for the functional Breuer-Major theorem are established.

\section{Preliminaries}
\label{sec:preliminaries}

\subsection{Probability on Hilbert spaces}
\label{sec:prob-hilb-spac}

Let $K$ be a separable real Hilbert space, $\mathcal{B}(K)$ the family
of Borel sets of $K$ and $(\Omega,\mathcal{F},P)$ a complete probability space.
A $K$-valued random variable $X$ is a measurable map from
$\left( \Omega,\mathcal{F} \right)$ to $(K,\mathcal{B}(K))$.
Such random variables are characterized by the property that for any continuous linear functional $\varphi \in K^{\ast}$, the function $\varphi(X) \colon
\Omega \to \R$ is a (real-valued) random variable. As usual, the distribution or law of a random variable $X$ is defined to be the push-forward probability measure $P \circ X^{-1}$ on $(K,\mathcal{B}(K))$. The set of all $K$-valued random variables is a vector space over the field of real numbers. If the Lebesgue integral
$  \E \left( \left\lVert X \right\rVert_K \right)
  =
  \int_{\Omega}^{} \left\lVert X \right\rVert_K \d{P}$
exists and is finite, then the Bochner integral $  \int_{\Omega}^{} X \d{P}$
exists in $K$ and is called the expectation of $X$. Slightly abusing
notation, we denote this integral by $\E \left( X \right)$ as well. It
will always be clear from the context if $\E \left( \cdot \right)$
denotes Lebesgue or Bochner integration with respect to $P$.
For $p \geq 1$, we denote by $L^p(\Omega;K)$ the Banach space of all equivalence classes (under almost sure equality) of $K$-valued random variables $X$ with finite $p$-th moment, i.e., such that
\begin{equation*}
  \left\lVert X \right\rVert_{L^p(\Omega,K)} = \left( \E \left( \left\lVert X \right\rVert_K^p \right) \right)^{1/p} < \infty.
\end{equation*}
Note that for all $X \in L^p(\Omega;K)$, the Bochner integral $\E
\left( X \right) \in B$ exists. If $X \in L^2(\Omega;K)$, the
covariance operator $S \colon K \to K$ of $X$ is defined by
\begin{equation}
 \label{eq:32}
  Su = \E \left( \left\langle X,u \right\rangle X \right).
\end{equation}
It is a positive, self-adjoint trace-class operator and verifies the identity
\begin{equation}
  \label{eq:16}
\on{tr} S = \E \left( \left\lVert X \right\rVert^2 \right).
\end{equation}
We denote by
$\mathcal{S}_1(K)$ the Banach space of all trace class operators on
$K$ with norm $\left\lVert T
\right\rVert_{\mathcal{S}_1(K)} = \on{tr} \left| T \right|$, where $\left| T \right| = \sqrt{TT^{\ast}}$. The subspace of
Hilbert-Schmidt operators will be denoted by $\on{HS}(K)$, its
inner product and associated norm by $\left\langle \cdot,\cdot
\right\rangle_{\on{HS}(K)}$ and $\left\lVert  \cdot
\right\rVert_{\on{HS}(K)}$, respectively. Recall that
\begin{equation*}
\norm{\cdot}_{\on{op}} \leq \norm{\cdot}_{\on{HS}(K)} \leq \norm{\cdot}_{\mathcal{S}_1},
\end{equation*}
where $\norm{\cdot}_{\on{op}}$ denotes the operator norm.

When there is no ambiguity about what Hilbert space $K$ underlies
$\left\langle \cdot , \cdot \right\rangle_K$, $\norm{\cdot}_K$, $\mathcal{S}_1(K)$ or $\on{HS}(K)$, we will drop the $K$ dependency
and just write $\left\langle \cdot , \cdot \right\rangle$,
$\norm{\cdot}$, $\mathcal{S}_1$, $\on{HS}$, and so on.

\subsection{Gaussian measures and Stein's method on abstract Wiener
  spaces}
\label{sec:gauss-meas-steins}
 In this section, we introduce Gaussian measures, the associated
 abstract Wiener spaces and Stein's method. We present the theory in a Banach space setting as specializing to Hilbert spaces brings no significant advantages at this point. Standard references for Gaussian measures and abstract Wiener spaces are the books~\cite{%
bogachev:1998:gaussian-measures%
,kuo:1975:gaussian-measures-banach%
}, Stein's method on abstract Wiener space has been introduced by Shih in~\cite{shih:2011:steins-method-infinite-dimensional}.

\subsubsection{Abstract Wiener spaces}
Let $H$ be a real separable Hilbert space with inner product
 $\left\langle \cdot, \cdot \right\rangle_H$ and define a norm
 $\norm{\cdot}_{}$ on $H$ (not necessarily induced by another inner product) that is weaker than $\norm{\cdot}_H$. Denote by $B$ the Banach space obtained as the
 completion of $H$ with respect to the norm $\norm{\cdot}_{}$ (note
 that if the $\norm{\cdot}_{}$ norm happens to be induced by an inner
 product, then $B$ is actually a Hilbert space), and
 define $i$ to be the canonical embedding of $H$ into $B$. Then, the
 triple $(i,H,B)$ is called an abstract Wiener space. We identify
 $B^{*}$ as a dense subspace of $H^{*}$ under the adjoint operator
 $i^{*}$ of $i$, so that we have the continuous embeddings $B^{*}
 \subset H \subset B$, where, as usual, $H$ is
 identified with its dual. The abstract Wiener measure $p$ on $B$ is characterized as the Borel measure on $B$ satisfying
\begin{equation*}
\int_B^{}e^{\on{i}\left\langle x,\eta \right\rangle_{B,B^{*}}}p(dx) = e^{-\frac{\norm{\eta}_{H}^2}{2}},
\end{equation*}
for any $\eta \in B^{*}$, where $\left\langle \cdot , \cdot \right\rangle_{B,B^{*}}$ denotes the dual pairing in $B$.

\subsubsection{Gaussian measures on Banach and Hilbert spaces}
\label{sec:gauss-meas-banach}
For a Banach space $B$, we denote by $\mathcal{B}(B)$ its family of Borel sets.
\begin{definition}
  \label{def:1}
  Let $B$ be a real separable Banach space. A Gaussian measure $\nu$
  is a  probability measure on $(B,\mathcal{B}(B))$, such that every
  linear functional $x \in B^{\ast}$, considered as a random variable on $(B,\mathcal{B}(B),\nu)$, has a Gaussian distribution (on $(\R,\mathcal{B}(\R))$). The Gaussian measure $\nu$ is called
\emph{centered} (or \emph{non-degenerate}), if these properties hold for the distributions of every $x \in B^{\ast}$.
\end{definition}
We see from the definition that every abstract Wiener measure is a Gaussian measure and, conversely, for any Gaussian measure $\nu$ on a separable Banach space $B$, there exists a Hilbert space $H$ such that the triple $(i,H,B)$ is an abstract Wiener space with associated abstract Wiener measure $\nu$ (see~\cite[Lemma 2.1]{kuelbs:1970:gaussian-measures-banach}). The space $H$ is called the Cameron-Martin space.

\subsubsection{Stein characterization of abstract Wiener measures}
Let $B$ be real separable Banach space with norm $\norm{\cdot}_{}$ and
let $Z$ be a $B$-valued random variable on some probability space
$(\Omega,\mathcal{F}, P)$ such that the distribution $\mu_Z$
of $Z$ is a non-degenerate Gaussian measure on $B$ with zero mean. Let
$(i,H,B)$ be the abstract Wiener space associated to the Wiener
measure $\mu_Z$, as described in the previous subsection. Let $\left\{
  P_t \colon t \geq 0 \right\}$ denote the Ornstein-Uhlenbeck semigroup associated with $\mu_Z$ and
defined, for any $\mathcal{B}(B)$-measurable
function $f$ and $x \in B$, by
\begin{equation*}
P_tf(x) = \int_B^{}f \left( e^{-t}x + \sqrt{1-e^{-2t}}y
\right)\mu_Z(dy),\quad t \geq 0,
\end{equation*}
provided such an integral exists. We have the following Stein lemma
for abstract Wiener measures (see \cite[Theorem 3.1]{shih:2011:steins-method-infinite-dimensional}).
\begin{theorem}
  \label{steinlemmabanachshi}
  Let $X$ be a $B$-valued random variable with distribution $\mu_X$.
  \begin{enumerate}
  \item[(i)] If $B$ is finite-dimensional, then $\mu_X = \mu_Z$ if and
    only if
\begin{equation}
\label{eq:3}
\Ex{\left\langle X,\nabla f(X) \right\rangle_{B,B^{*}} - \Delta_G f(X)} = 0
\end{equation}
for any twice differentiable function $f$ on $B$ such that
$\Ex{\norm{\nabla^2f(Z)}_{\mathcal{S}_1(H)}}<\infty$.
\item[(ii)] If $B$ is infinite-dimensional, then $\mu_X = \mu_Z$ if
  and only if \eqref{eq:3} holds for any twice $H$-differentiable
  function $f$ on $B$ such that $\nabla f(x) \in B^{*}$ for
  any $x \in B$, $\Ex{\norm{\nabla^2f(Z)}_{\mathcal{S}_1(H)}}<\infty$ and
  $\Ex{\norm{\nabla f(Z)}_{B^{*}}^2} < \infty$.
  \end{enumerate}
\end{theorem}

The notion  of an $H$-derivative appearing in Theorem \ref{steinlemmabanachshi} was
introduced by Gross in~\cite{gross:1967:potential-theory-hilbert} and is defined as
follows. A function $f \colon U \to W$ from an open set $U$ of
$B$ into a Banach space $W$ is said to be $H$-differentiable at a
point $x \in U$ if the map $\phi(h) = f(x+h)$, $h \in H$, regarded as
a function defined in a neighborhood of the origin of $H$ is Fr\'echet
differentiable at 0. The Fr\'echet derivative $\phi'(0)$ at $0 \in H$
is called the $H$-derivative of $f$ at $x \in B$. The $H$-derivative
of $f$ at $x$ in the direction $h \in H$ is denoted by $\left\langle
  \nabla f(x),h \right\rangle_H$. The $k$-th order $H$-derivatives of
$f$ at $x$ can be defined inductively and are denoted by $\nabla^k
f(x)$ for $k \geq 2$, provided they exist. If $f$ is scalar-valued,
$\nabla f(x) \in H^{*} \approx H$ and $\nabla^2 f(x)$ is regarded as a
bounded linear operator from $H$ into $H^{*}$ for any $x \in U$, and
the notation $\left\langle \nabla^2 f(x) h, k \right\rangle_H$ stands
for the action of the linear form $\nabla^2 f(x)(h, \cdot)$, $h \in H$, on $k \in
H$, denoted by $\nabla^2 f(x)(h, k)$. Furthermore, if $\nabla^2 f(x)$ is
a trace-class operator on $H$, we can define the so-called Gross
Laplacian $\Delta_G f(x)$ of $f$ at $x$ appearing in \eqref{eq:3} by $\Delta_G
f(x) = \on{tr}_H(\nabla^2f(x))$.
\begin{remark}[On the relation between Fr\'echet and $H$-derivatives]
An $H$-derivative $\nabla f(x)$ at $x \in B$ determines an element in
$B^{*}$ if there is a constant $C >0$ such that $\abs{\left\langle
    \nabla f(x),h \right\rangle_H} \leq C \norm{h}_H$ for any $h \in
H$. Then, $\nabla f(x)$ defines an element of $B^{*}$ by continuity
and we denote this element by $\nabla f(x)$ as well. Now, if $f$ is
also twice Fr\'echet differentiable on $B$, then $\nabla f(x)$
coincides with the first-order Fr\'echet derivative $f'(x)$ at $x \in
B$ and is automatically in $B^{*}$. Furthermore, $\nabla^2 f(x)$
coincides with the restriction of the second-order Fr\'echet
derivative $f''(x)$ to $H \times H$ at $x \in B$. In this framework,
since for any $x \in B$, $f''(x)$ is a bounded linear operator from
$B$ into $B^{*}$, Goodman's theorem (see~\cite[Chapter 1, Theorem 4.6]{kuo:1975:gaussian-measures-banach}) implies that $\nabla^2
f(x)$ is a trace-class operator on $H$ and that, consequently, the Gross Laplacian
$\Delta_G f(x)$ is well-defined. Twice Fr\'echet differentiability
hence constitutes a sufficient condition for the existence of the Gross Laplacian.   
\end{remark}

\subsubsection{Stein's equation and its solutions for abstract Wiener
  measures}
\label{sec:steins-equation-its}
In view of the above Stein lemma (Theorem \ref{steinlemmabanachshi}), the associated
Stein equation is given by
\begin{equation}
  \label{eq:4}
\Delta_Gf(x) - \left\langle x,\nabla f(x) \right\rangle_{B,B^{*}} = h(x) -
\Ex{h(Z)}, \qquad x \in B,
\end{equation}
where $h$ is given in some class of test functionals. Shih showed in~~\cite{shih:2011:steins-method-infinite-dimensional} that
\begin{equation}
  \label{steinsolus}
f_h(x) = -\int_0^{\infty} \left( P_uh(x) - \Ex{h(Z)}
\right)du,\qquad x\in B
\end{equation}
solves the Stein equation \eqref{eq:4} whenever $h$ is an element of
$\on{ULip-1}(B)$, the Banach space of scalar-valued
uniformly 1-Lipschitz functions $h$ on $B$ with the norm $\norm{h}_{} =
\norm{h}_{\on{ULip}} + \abs{h(0)}$, where
\begin{equation*}
\norm{h}_{\on{ULip}} = \sup_{x \neq y \in B} \frac{\abs{h(x)-h(y)}}{\norm{x-y}_{}}<\infty.
\end{equation*}

In what follows, we will consider test functions from the space
$C^k_b(K)$ of real-valued, $k$-times Fr\'echet differentiable
functions on a separable Hilbert space $K$ with bounded derivatives up to order $k$. A function $h$
thus belongs to $C^k_b(K)$ whenever
\begin{equation*}
\norm{h}_{C^k_b(K)} = \sup_{j=1, \ldots, k} \sup_{x \in K} \norm{\nabla^j
  h(x)}_{K^{\otimes j}} < \infty.
\end{equation*}
The following lemma collects some properties of the Stein solution $f_h$ for a given function $h \in C^k_b(K)$.
\begin{lemma}
  \label{lem:4}
  Let $K$ be a separable Hilbert space, $k \geq 1$ and $h \in C^k_b(K)$. Then the Stein solution $f_h$ defined in \eqref{steinsolus} also belongs to $C^k_b(K)$ and furthermore one has that
\begin{equation}
  \label{eq:8}
  \sup_{u \in K}
  \norm{\nabla^{j} f_h(u)}_{K^{\otimes j}}
  \leq
  \frac{1}{j}
  \left\lVert h \right\rVert_{C^j_{b}(K)}, \qquad j \in \N, \qquad j \leq k.
\end{equation}

\end{lemma}
\begin{proof}
As for any $x \in K$, $f_h(x) = -\int_0^{\infty} \left( P_uh(x) - \Ex{h(Z)}
\right)du$, we have, for any $j = 1,\ldots , k$,
\begin{equation*}
\nabla^j f_h(x) = -\int_0^{\infty} \nabla^j P_uh(x)du,
\end{equation*}
so that
\begin{align*}
\norm{f_h}_{C^k_b(K)} &= \sup_{j=1, \ldots, k} \sup_{x \in K}
                        \norm{-\int_0^{\infty} \nabla^j P_uh(x)du }_{K^{\otimes j}} \\
  & \leq \sup_{j=1, \ldots, k} \sup_{x \in K}
                        \int_0^{\infty} \norm{\nabla^j
    P_uh(x)}_{K^{\otimes j}} du.
\end{align*}
Using the property of the semigroup $P$ that $\nabla^j
P_u h(x) = e^{-j u}P_u \nabla^j h(x)$, and the fact
that $P$ is contractive yields
\begin{align*}
\norm{f_h}_{C^k_b(K)} & \leq \sup_{j=1, \ldots,
                        k} \sup_{x \in K}  \int_0^{\infty} e^{-ju} \norm{P_u \nabla^j
                        h(x)}_{K^{\otimes j}} du \\
  & \leq \sup_{j=1, \ldots,
                        k} \sup_{x \in K} \int_0^{\infty} e^{-ju} \norm{\nabla^j
    h(x)}_{K^{\otimes j}} du \\
  &= \sup_{j=1, \ldots,
                        k} \sup_{x \in K} \frac{1}{j} \norm{\nabla^j
    h(x)}_{K^{\otimes j}} \\
  & \leq \norm{h}_{C^k_b(K)} < \infty,
\end{align*}
proving that $f_{h} \in C^k_b(K)$. The bound~\eqref{eq:8} can be derived similarly.
\end{proof}

\subsection{Dirichlet structures}
\label{sec:dirichl-struct-hilb}
In this section, a Dirichlet structure for Hilbert-valued random
variables is introduced, which will be the framework we work in.
We start by recalling the well-known definition in the case of real-valued random variables (full details can for example be found in~\cite{bouleau-hirsch:1991:dirichlet-forms-analysis%
  ,fukushima-oshima-takeda:2011:dirichlet-forms-symmetric%
  ,ma-rockner:1992:introduction-theory-nonsymmetric%
  ,bakry-gentil-ledoux:2014:analysis-geometry-markov}, where the
latter reference emphasizes the equivalent notion of a Markov triple).
Given a probability space $(\Omega,\mathcal{F},P)$, a Dirichlet structure $(\mathbb{D},\mathcal{E})$ on $L^2(\Omega;\mathbb{R})$ with associated carré du champ operator $\Gamma$ consists of a Dirichlet domain $\mathbb{D}$, which is a dense subset of $L^2(\Omega;\mathbb{R})$ and a carré du champ operator $\Gamma \colon \mathbb{D} \times \mathbb{D} \to L^1(\Omega;\mathbb{R})$ characterized by the following properties.

\begin{itemize}
\item[-] $\Gamma$ is bilinear, symmetric ($\Gamma(f,g)=\Gamma(g,f)$) and positive $\Gamma(f,f) \geq0$,
\item[-] for all $m,n \in \N$, all Lipschitz and continuously differentiable functions $\varphi \colon \R^m \to \R$ and $\psi \colon \R^n \to \R$ and all $f=(f_1,\dots,f_m) \in \mathbb{D}^m$, $g=(g_1,\dots,g_n) \in \mathbb{D}^n$, it holds that
\begin{equation}
 \label{eq:15}
  \Gamma(\varphi(f),\psi(g)) = \sum_{i=1}^m \sum_{j=1}^{n} \partial_i \varphi(f) \partial_j \psi(g) \Gamma(f_i,g_j),
\end{equation}
\item[-] the induced positive linear form $f \mapsto \mathcal{E}(f,f)$, where
\begin{equation*}
  \mathcal{E}(f,g) = \frac{1}{2} \E \left( \Gamma(f,g) \right)
\end{equation*}
is closed in $L^2(\Omega;\mathbb{R})$, i.e., $\mathbb{D}$ is complete when equipped with the norm
\begin{equation*}
  \left\lVert \cdot \right\rVert_{\mathbb{D}}^2 = \left\lVert \cdot \right\rVert_{L^2(\Omega;\R)}^2 + \mathcal{E}(\cdot).
\end{equation*}
\end{itemize}
Here and in the following, $\E \left( \cdot \right)$ denotes
expectation on $(\Omega,\mathcal{F})$ with respect to $P$.
The form $f \to \mathcal{E}(f,f)$ is called a Dirichlet form, and, as is
customary, we will write $\mathcal{E}(f)$ for
$\mathcal{E}(f,f)$. Every Dirichlet form gives rise to a strongly
continuous semigroup $\left\{ P_t \colon t\geq 0 \right\}$ on $L^2(\Omega;\mathbb{R})$ and an associated symmetric Markov generator $-L$, defined on a dense subset $\on{dom}(-L) \subseteq \mathbb{D}$. We will often switch between $-L$ and $L$, as these two operators only differ by sign.
There are two important relations between $\Gamma$ and $L$. The first one is the integration by parts formula
\begin{equation}
 \label{eq:24}
  \E \left( \Gamma(f,g)  \right) = - \E \left( f Lg  \right) = - \E \left( g L f \right),
\end{equation}
valid whenever $f,g \in \mathbb{D}$, the second one is the relation
\begin{equation*}
  \Gamma(f,g) = \frac{1}{2} \left( L(fg) - gLf-fLg \right),
\end{equation*}
which holds for all $f,g\in\on{dom}(L)$ such that $fg \in \on{dom}(L)$.
If $-L$ is diagonalizable with spectrum $\left\{ 0=\lambda_0<\lambda_1 < \dots \right\}$, a pseudoinverse $-L^{-1}$ can be introduced via spectral calculus as follows: if $f = \sum_{i=0}^{\infty} f_i$ with $f_i \in \ker \left( L + \lambda_i \on{Id} \right)$, then
\begin{equation*}
  -L^{-1} f = \sum_{i=1}^{\infty} \frac{1}{\lambda_i} f_i.
\end{equation*}
It follows that
\begin{equation*}
  LL^{-1}f = f - \E \left( f \right).
\end{equation*}

Consider now such a Dirichlet structure on $L^2(\Omega;\mathbb{R})$ with diagonalizable generator as given and denote the Dirichlet domain, Dirichlet form, carré du champ operator, its associated infinitesimal generator and pseudo-inverse by $\widetilde{\mathbb{D}}$, $\widetilde{\mathcal{E}}$, $\widetilde{\Gamma}$, $\widetilde{L}$ and $\widetilde{L}^{-1}$, respectively, in order to distinguish these objects from their extensions to the Hilbert-valued setting to be introduced below.

Given a separable Hilbert space $K$, one has that $L^2(\Omega;K)$ is isomorphic to $L^2(\Omega;\R)\otimes K$. The Dirichlet structure on $L^2(\Omega;\R)$ can therefore be extended to $L^2(\Omega;K)$ via a tensorization procedure as follows.

Let $\left\{ 0=\lambda_0 < \lambda_1 < \dots \right\}$ be the spectrum of $-\widetilde{L}$ and $\mathcal{A}$ the set of all functions $F$ of the form
\begin{equation}
  \label{eq:34}
  F = \sum_{(i,j) \in I} f_{p_i} \otimes k_j,
\end{equation}
where $I \subseteq \N \times \N$ is finite, the $f_{p_i}$ are eigenfunctions of $-\widetilde{L}$ with eigenvalue $\lambda_{p_i} \geq 0$ and the $k_j$ form an orthonormal basis in $K$. Then $\mathcal{A}$ is dense in $L^2(\Omega;K)$.
For $F \in \mathcal{A}$ of the form~\eqref{eq:34} and analogously $G=\sum_{(i',j') \in I'}^{} f_{p_{i'}} \times k_{j'} \in \mathcal{A}$ , define linear operators $L$, $L^{-1}$ by
\begin{align*}
LF &= \sum_{(i,j) \in I} (\widetilde{L}f_{p_i}) \otimes k_j = -
      \sum_{(i,j) \in I} \lambda_{p_i} f_{p_i} \otimes k_j,
  \\
  L^{-1}(F) &=
              \sum_{(i,j) \in I}  (\widetilde{L}^{-1}f_{p_i}) \otimes k_j =
              -
              \sum_{\substack{(i,j) \in I\\p_i \neq 0}} \frac{1}{\lambda_{p_i}} f_{p_i} \otimes k_j,
\end{align*}
a bilinear and positive operator $\Gamma$ by
\begin{equation*}
\Gamma(F,G) = \frac{1}{2}
 \sum_{(i,j)\in I} \sum_{(i',j') \in I'}^{}
 \widetilde{\Gamma}(f_{p_i},f_{p_{i'}})
 \otimes
 \left(
   k_{j} \otimes k_{j'}
   +
   k_{j'} \otimes k_{j}
 \right)
\end{equation*}
and a bilinear, positive and symmetric form $\mathcal{E}$ by
\begin{equation*}
\mathcal{E}(F,G) = \E \left( \on{tr} \Gamma(F,G) \right),
\end{equation*}
where in the definition of $\mathcal{E}$ we identify $\Gamma(F,G) \in L^2(\Omega;\R)\otimes K \otimes K \simeq L^2(\Omega;\mathcal{L}(K,K))$ with a random operator on $K$, whose action is given by
\begin{equation*}
  \Gamma(F,G)u =
  \frac{1}{2}
  \sum_{(i,j) \in I} \sum_{(i',j') \in I'}^{}  \widetilde{\Gamma}(f_{p_i},f_{p_{i'}})
  \left(
    \left\langle k_j,u \right\rangle \otimes k_{j'}
    +
    \left\langle k_{j'},u \right\rangle \otimes k_{j}
    \right), \quad u \in K.
  \end{equation*}
For all $F,G \in \mathcal{A}$, the operator  $\Gamma(F,G)$ is then of trace class and an element of
$L^1(\Omega;\mathcal{S}_1)$. It is standard to verify that the
definitions of $L$, $L^{-1}$ and $\Gamma$ do not depend on the choice
of the basis of $K$. Furthermore, from the well-known results for $\widetilde{L}$, $\widetilde{\Gamma}$ and $\widetilde{\mathcal{E}}$, we can extend them as follows.

\begin{proposition}
  \label{prop:2}
  The operators $L$, $L^{-1}$, $\mathcal{E}$ and $\Gamma$ introduced above can be extended to $\on{dom} (L)$, $\on{dom}(L^{-1})$ and $\on{dom}(\Gamma)= \on{dom}(\mathcal{E}) = \mathbb{D} \times \mathbb{D}$, given by
  \begin{align*}
    \on{dom}(L)
    &=
    \left\{
    F \in L^2(\Omega;K)
    \colon
    \sum_{p=1}^{\infty} \lambda_p^2 \,  \widetilde{\pi}_p \left( \left\lVert F \right\rVert^2 \right)
    <
    \infty
      \right\}
    \\
    \on{dom}(L^{-1}) &= L^2(\Omega;K)
    \\
    \intertext{and}
    \mathbb{D} &=
    \left\{
    F \in L^2(\Omega;K)
    \colon
    \sum_{p=1}^{\infty} \lambda_p \,  \widetilde{\pi}_p \left( \left\lVert F \right\rVert^2 \right)
    <
    \infty
      \right\},
  \end{align*}
respectively, where $\widetilde{\pi}_p$ denotes the orthogonal
projection onto
\begin{equation*}
\ker (\widetilde{L} + \lambda_p \on{Id} ) \subseteq L^2(\Omega;\R).
\end{equation*}
In particular, one has
\begin{equation*}
 \mathcal{A} \subseteq  \on{dom}(L) \subseteq \mathbb{D} \subseteq \on{dom}(L^{-1}) = L^2(\Omega;K),
\end{equation*}
where all inclusions are dense.
\end{proposition}

Throughout this article, the extensions of $L$, $L^{-1}$ and $\Gamma$ to their maximal domains will still be denoted by the same symbols. The operators just defined yield a Dirichlet structure $(\Gamma,\mathbb{D})$ on $L^2(\Omega;K)$, which is a natural counterpart to the given structure $(\widetilde{\Gamma},\widetilde{\mathbb{D}})$ on $L^2(\Omega;\R)$. The following theorem summarizes its main features.

\begin{theorem}
  \label{thm:8}
  For a Dirichlet structure $(\mathbb{D},\Gamma)$ on $L^2(\Omega;K)$, consisting of a dense subspace $\mathbb{D}$ of $L^2(\Omega;K)$ and a carré du champ operator $\Gamma \colon \mathbb{D} \times \mathbb{D} \to L^{1}(\Omega;\mathcal{S}_1)$ as introduced above, the following is true.
\begin{enumerate}[(i)]
\item $\Gamma$ is bilinear, almost surely positive (i.e., $\Gamma(F,F) \geq 0$ as an operator on $K$), symmetric in its arguments and self-adjoint  ($\langle\Gamma(F,G)u,v\rangle=\langle u,\Gamma(F,G)v \rangle$ for all $u,v \in K$).
\item The Dirichlet domain $\mathbb{D}$, endowed with the norm
\begin{equation*}
  \left\lVert F \right\rVert_{\mathbb{D}}^2
  =
  \left\lVert F \right\rVert_{L^2(\Omega;K)}^2 + \left\lVert \Gamma(F,F) \right\rVert_{L^1(\Omega;\mathcal{S}_1)}
\end{equation*}
is complete, so that $\Gamma$ is closed.
\item
  For all Lipschitz and Fréchet differentiable operators $\varphi, \psi$ on $K$ and $F,G \in \mathbb{D}$, one has that $\varphi(F), \psi(G) \in \mathbb{D}$ and the  diffusion identity
\begin{equation}
 \label{eq:40}
 \Gamma(\varphi(F),\psi(G)) =
 \frac{1}{2}
 \big(
 \nabla \varphi(F)^{\ast}
 \Gamma(F,G)
 \nabla \psi(G)
 +
 \nabla \psi(G)^{\ast}
 \Gamma(F,G)
  \nabla \varphi(F)
 \big)
\end{equation}
holds,  where $\nabla \varphi(F)$ and $\nabla \psi(G)$ denote the Fréchet derivatives of $\varphi$ and $\psi$ at $F$ and $G$, respectively, and $\nabla \varphi(F)^{\ast}$, $\nabla \psi(G)^{\ast}$ are their adjoints in $K$.
\item The associated generator $-L$ acting on $L^2(\Omega;K)$ is positive, symmetric, densely defined and has the same spectrum as $-\widetilde{L}$.
\item There exists a compact pseudo-inverse $L^{-1}$ of $L$ such that
\begin{equation*}
  LL^{-1}F = F - \E \left( F \right)
\end{equation*}
for all $F \in L^2(\Omega;K)$, where the expectation on the right is a
Bochner integral (well defined as $F \in L^2(\Omega;K)$).
\item The integration by parts formula
  \begin{equation}
 \label{eq:49}
   \E \left( \on{tr} \Gamma(F,G) \right)
 =
 - \E \left( \left\langle LF,G \right\rangle \right)
 =
 -
 \E \left( \left\langle F,LG \right\rangle \right).
\end{equation}
is satisfied for all $F,G \in \on{dom}(-L)$.
\item The carré du champ $\Gamma$ and the generators $L$ and $\widetilde{L}$ are connected through the identity
\begin{equation*}
   \on{tr} \Gamma(F,G)
 =
 \frac{1}{2}
\left(
   \widetilde{L}
   \left\langle F,G \right\rangle
   -
   \left\langle LF,G \right\rangle
   -
   \left\langle F,LG \right\rangle
    \right),
  \end{equation*}
  valid for $F,G \in \on{dom}(L)$.
\item The fundamental identity
\begin{equation}
 \label{eq:42}
   \left\langle \Gamma(F,G)u,v \right\rangle
 =
 \frac{1}{2}
 \left(
   \widetilde{\Gamma} \left(
     \left\langle F,u \right\rangle
     ,
     \left\langle G,v \right\rangle
   \right)
   +
   \widetilde{\Gamma} \left(
     \left\langle G,u \right\rangle
     ,
     \left\langle F,v \right\rangle
   \right)
 \right),
\end{equation}
connecting $\Gamma$ and its one-dimensional counterpart $\widetilde{\Gamma}$ is valid for all $F,G\in \mathbb{D}$ and all $u,v \in K$.
\end{enumerate}
\end{theorem}

\begin{proof}
  Parts $(i)-(ii)$ and $(iv)-(viii)$ are straightforward to verify. In order to prove $(iii)$, write
\begin{equation*}
 F = \sum_{p=0}^{\infty} \sum_{i=1}^{\infty} f_p \otimes k_i  \qquad \text{ and } \qquad
 G = \sum_{p=0}^{\infty} \sum_{i=1}^{\infty} g_p \otimes k_i,
\end{equation*}
where the $f_p$ and $g_p$ are eigenfunctions of $\widetilde{L}$ with eigenvalue $-\lambda_p$,  and $\left\{ k_i \colon i \in \N \right\}$ is an orthonormal basis of $K$. Let $K_{n}=\on{span} \left\{ k_i \colon 1 \leq i \leq n \right\}$ and $\rho_{n}$ be the orthogonal projection onto $L^2(\Omega;K_{n})$, so that
\begin{equation*}
\rho_n(F)= \sum_{p=0}^{\infty} \sum_{i=1}^{n} f_p \otimes k_i  \qquad \text{ and } \qquad
 \rho_{n}(G) = \sum_{p=0}^{\infty} \sum_{i=1}^{n} g_p \otimes k_i.
\end{equation*}
Denote by $\mathfrak{i}_n \colon K_n \to \R^n$ the canonical isometric isomorphism mapping $K_n$ to $\R^n$ so that $\xi_n = \mathfrak{i}_n \circ \rho_n(F) \in \R^n$ and $\upsilon_n = \mathfrak{i}_n \circ \rho_n(G) \in \R^n$.

Let $\widetilde{\varphi}_n = \varphi \circ \mathfrak{i}^{-1}_n$ and $\widetilde{\psi}_n=\psi \circ \mathfrak{i}^{-1}_n$. Then  $\widetilde{\varphi}_n \colon \R^n \to K$ is Lipschitz and Fréchet differentiable, with Fréchet derivative given by
\begin{equation*}
  \nabla \widetilde{\varphi}_n(x) (y)
  =
  \nabla\varphi (\mathfrak{i}_n^{-1}(x)) (\mathfrak{i}^{-1}_n(y))
\end{equation*}
for all $x,y \in \R^n$ and an analogous result is true for $\nabla \widetilde{\psi}_n$.
Therefore, via
\begin{equation*}
  \Gamma \left( \varphi(\rho_n(F)), \psi(\rho_n(G)) \right)
=
\Gamma \left( \widetilde{\varphi}_n(\xi_n),\widetilde{\psi}_n(\upsilon_n)\right)
\end{equation*}
and identity~\eqref{eq:42}, the assertion can be transformed into an equivalent assertion for $\widetilde{\Gamma}$, which can then be verified by  tedious but straightforward calculations, using the diffusion property~\eqref{eq:15} for $\widetilde{\Gamma}$ and then letting $n \to \infty$.

\end{proof}

The most important example in our context is the Dirichlet structure
given by the Ornstein-Uhlenbeck generator of a Hilbert-valued
Ornstein-Uhlenbeck semigroup. Here, $-L=\delta D$, where $D$ and
$\delta$ denote the Malliavin derivative and divergence operator, and
the carré du champ operator is given by $\Gamma(X,Y) = \left\langle
  DX,DY \right\rangle_{\mathfrak{H}}$, where $\mathfrak{H}$ is the
Hilbert space associated to the underlying isonormal Gaussian process (see Section~\ref{sec:hilb-valu-wien} for full details). The corresponding eigenspaces are known as Wiener chaos and spanned by the infinite-dimensional Hermite polynomials. In the same way, one can obtain Jacobi, Laguerre or other polynomial chaoses (see for example~\cite{azmoodeh-campese-poly:2014:fourth-moment-theorems} for the real-valued case). We refer to the monographs quoted at the beginning of this section for further numerous examples.

\section{Approximation of Hilbert-valued Gaussians}
\label{sec:appr-hilb-valu}

In this section, we combine Stein's method introduced in Section~\ref{sec:gauss-meas-steins} with the Dirichlet structure defined in Section~\ref{sec:dirichl-struct-hilb} in order to derive bounds on a probabilistic metric between the laws of square integrable random variables and a Gaussian, both taking values in some separable Hilbert space.

Throughout the whole section, this separable Hilbert space will be denoted by $K$, and we furthermore assume as given a Dirichlet structure on $L^2(\Omega;K)$
as introduced in the previous section, with Dirichlet domain $\mathbb{D}$, carré du champ operator $\Gamma$ and associated generator $L$.

The probabilistic distance we use is the well-known $d_2$-metric, given by
\begin{equation}
 \label{eq:45}
 d_2(F,G) =
 \sup_{
   \substack{
     h \in C_b^2(K)
     \\ \left\lVert h \right\rVert_{C_b^2(K)} \leq 1
   }
 }
 \left| \E \left( h(F) \right)
   -
   \E \left( h(G) \right)
 \right|,
\end{equation}
where $C_b^2(K)$ denotes the twice Fréchet differentiable, real-valued functions on $K$ with uniformly bounded first and second derivatives (see Section~\ref{sec:steins-equation-its}). In an infinite-dimensional context, this distance has already been used in~\cite{coutin-decreusefond:2013:steins-method-brownian} and, in a weakened form, also in~\cite{barbour:1990:steins-method-diffusion}. As already observed in~\cite{coutin-decreusefond:2013:steins-method-brownian}, it metrizes convergence in distribution:

\begin{lemma}
  \label{lem:8}
  If $\left\{ F_n \colon n \in \N_0 \right\}$ is a sequence of $K$-valued random variables such that
\begin{equation*}
  d_{2}(F_n,F_0) \to 0
\end{equation*}
as $n \to \infty$, then the law of $F_n$ converges in distribution to the law of $F_0$, i.e.,
\begin{equation*}
  \E \left( h(F_n) \right) \to \E \left( h(F_0) \right)
\end{equation*}
as $n \to \infty$, for all bounded, real-valued and continuous functions $h$ on $K$.
\end{lemma}

\begin{proof}
  The proof given in~\cite[Lemma 4.1]{coutin-decreusefond:2013:steins-method-brownian} for $K=\ell^2(\N)$  continues to work without any modification.
\end{proof}

\subsection{An abstract carr\'e du champ bound}
\label{sec:an-abstract-carre}

The following general bound between the laws of a square integrable
$K$-valued random variable in the Dirichlet domain $\mathbb{D}$ and an arbitrary
Gaussian random variable holds.

\begin{theorem}
  \label{thm:1}
  Let $Z$ be a centered, non-degenerate Gaussian random variable on
  $K$ with covariance operator $S$ and let $F \in \mathbb{D}$. Then
  \begin{align}
    \label{eq:23}
    d_2 (F,Z)
    &\leq
      \frac{1}{2}
      \left\lVert
          \Gamma(F,-L^{-1}F) - S
      \right\rVert_{L^2(\Omega;\on{HS})}.
  \end{align}
  If $K$ has dimension $d<\infty$, then
\begin{equation}
 \label{eq:53}
    d_W (F,Z)
    \leq
    C_{S,d}
      \left\lVert
          \Gamma(F,-L^{-1}F) - S
      \right\rVert_{L^2(\Omega;\on{HS})},
\end{equation}
where $d_W$ denotes the Wasserstein distance, and
\begin{equation*}
  C_{S,d} =
 \sqrt{d   \left\lVert S \right\rVert_{\on{op}}}
    \left\lVert S^{-1} \right\rVert_{\on{op}}.
\end{equation*}
\end{theorem}
\begin{proof}
  To prove~\eqref{eq:23}, it suffices to show that  for $h \in C_b^2(K)$ one has
\begin{multline}
  \label{eq:58}
  \left|   \E \left( \tr_H \nabla^2 f_h(F)
    -
    \left( F,\nabla f_h(F) \right)_{K,K^{\ast}}
  \right) \right|
  \\ \leq
  \frac{1}{2}
  \left\lVert \nabla^2 h \right\rVert_{C_b^2(K)}
\left\lVert
          \Gamma(F,-L^{-1}F) - S
      \right\rVert_{L^2(\Omega;\on{HS})},
\end{multline}
    where $f_h$ is the Stein solution given by \eqref{steinsolus}.
Indeed, using Stein's equation (see~\eqref{eq:4}), the left hand side of~\eqref{eq:58} is equal to $
\left|
    \E \left( h(X) \right)
    -
    \E \left( h(Z) \right)
  \right|$, so that the assertion follows after taking the supremum over $h$.

Identifying $K^{\ast}$ with $K$, using the integration by parts formula~\eqref{eq:49} and the diffusion property~\eqref{eq:40} for the carré du champ, we can write
  \begin{align*}
    \E \left(
    \left( F,\nabla f_h(F) \right)_{K,K^{\ast}} \right)
    &=
  \E \left( \left\langle F,  \nabla f_h(F)  \right\rangle_{K} \right)
  \\ &=
    \E \left( \left\langle  LL^{-1}F,\nabla f_h(F)  \right\rangle_{K} \right)
  \\ &=
       \E \left( \on{tr}_K \Gamma
       \left(\nabla f_h(F),-L^{-1}F  \right) \right)
         \\ &
    =
  \E \left(
  \on{tr}_K
  \left(
  \nabla^2f_h(F)
  \Gamma(F, -L^{-1} F)
   \right)
 \right).
\end{align*}
  Now let $H$ be the Cameron-Martin space associated to $Z$ as introduced in
  Section~\ref{sec:gauss-meas-steins}. As
  the covariance operator $S$ of $Z$ is compact and one-to-one, it holds that $S = \sum_{i \in \N}^{} \lambda_i \left\langle \cdot ,
    e_k \right\rangle_K e_i$ for some $\lambda_i>0$ and an orthonormal
  basis $\left\{ e_i \colon i \in \N \right\}$ of $H$ consisting of eigenvectors. Then $\left\{ k_i \colon i \in \N \right\}$, where $k_{i}=\frac{1}{\sqrt{\lambda_i}} e_i$, is an orthonormal basis of $K$, as $H=\sqrt{S}(K)$. It thus follows that
\begin{equation*}
  \on{tr}_H \nabla^2 f_h(F)
  =
    \sum_{i\in \N}^{} \nabla^2 f_h(F)(e_i,e_i)
\\  =
       \sum_{i\in \N}^{} \nabla^2 f_h(F)\left( S \, k_i, k_i \right)
=
       \on{tr}_K \left( \nabla^2 f_h(F) S \right).
\end{equation*}
Combining the last two calculations yields that
\begin{equation*}
 \E \left(
   \left\langle F,\nabla f_h(F) \right\rangle_{K,K^{\ast}}
   -
   \tr_H \nabla^2 f_h(F)
 \right)
 =
 \E \left(
   \tr_K \left( \nabla^2 f_h(F)
     \left( \Gamma(F,-L^{-1}F) - S \right)
   \right)
 \right),
\end{equation*}
and, taking absolute values and applying Hölder's inequality for the Schatten norms, we get
\begin{align}
&\left\lvert
  \E \left(
  \left\langle F,\nabla f_h(F) \right\rangle_{K,K^{\ast}}
-
   \tr_H \nabla^2 f_h(F)
 \right)
  \right\rvert \notag
 \\  &\qquad\qquad\qquad\qquad\quad = \notag
       \left|
        \E \left(
   \tr_K \left( \nabla^2 f_h(F)
     \left( \Gamma(F,-L^{-1}F) - S \right)
   \right)
 \right)
       \right|
  \\ & \qquad\qquad\qquad\qquad\quad \leq \notag
       \E
       \left(
       \tr_K
        \left|
         \nabla^2 f_h(F)
     \left( \Gamma(F,-L^{-1}F) - S \right)
       \right| \right)
  \\ & \qquad\qquad\qquad\qquad\quad \leq \notag
       \E \left(
       \left\lVert
       \nabla^2 f_h(F)
       \right\rVert_{\on{HS}(K)}
       \left\lVert
       \Gamma(F,-L^{-1}F) - S
       \right\rVert_{\on{HS}(K)}
       \right)\notag
  \\ &\qquad\qquad\qquad\qquad\quad \leq \label{eq:41}
       \left\lVert
       \nabla^2 f_h(F)
       \right\rVert_{L^{2}(\Omega;\on{HS}(K))}
       \left\lVert
       \Gamma(F,-L^{-1}F) - S
       \right\rVert_{L^{2}(\Omega;\on{HS}(K))}.
\end{align}
An application of Lemma~\ref{lem:4} now yields \eqref{eq:58}, finishing the proof of~\eqref{eq:23}. If $K$ has finite-dimension $d$, one can proceed
as in~\cite[Proposition 4.3.2]{nourdin-peccati:2012:normal-approximations-malliavin} to obtain that
\begin{equation*}
    \left\lVert
       \nabla^2 f_h(F)
     \right\rVert_{L^{2}(\Omega;\on{HS}(K))}
     \leq
     C_{S,d} \left\lVert h \right\rVert_{\on{Lip}},
\end{equation*}
where $\left\lVert \cdot \right\rVert_{\on{Lip}}$ denotes the Lipschitz norm.
The Wasserstein distance is then obtained by approximating Lipschitz functions in $C_b^2(K)$ (for example by convoluting a Gaussian kernel).
\end{proof}

If $Z$ is a $K$-valued Gaussian random variable with covariance operator $S$, then, taking $L$ to be the Ornstein-Uhlenbeck generator (see the forthcoming Section~\ref{sec:hilb-valu-wien}), one has that $\Gamma(Z,-L^{-1}Z) = S$. Therefore, taking $F$ to be Gaussian in Theorem~\ref{thm:1} yields a bound on the distance between two Gaussians $Z_1$, $Z_2$ in terms of the Hilbert-Schmidt norm of their covariance operators $S_1$, $S_2$. We state this as a corollary.
\begin{corollary}
  \label{cor:1}
  Let $Z_1, Z_2$ be two centered, non-degenerate Gaussian random variables on $K$ with covariance
operators $S_1, S_2$, respectively. Then, it holds that
  \begin{equation*}
    d_2 (Z_1,Z_2)
    \leq
    \frac{1}{2}
    \left\lVert
          S_1 -S_2 \right\rVert_{\on{HS}}.
  \end{equation*}
\end{corollary}

We continue with some remarks on Theorem~\ref{thm:1}.

\begin{remark}
  \label{rmk:2}\hfill
  \begin{enumerate}[(i)]
  \item Note that the proof of Theorem~\ref{thm:1} does not use  diagonalizability of $L$, so that this assumption can be replaced by weaker conditions guaranteeing that a pseudo-inverse can still be defined (in a finite-dimensional context, this has been done in~\cite{bourguin-campese-leonenko-ea:2019:four-moments-theorems}). However, we will not need this level of generality.
\item While $\left\lVert \Gamma(F,-L^{-1}F) - S
      \right\rVert_{\on{HS}}$ is almost surely finite for any $F \in \mathbb{D}$, it might be that $\left\lVert \Gamma(F,-L^{-1}F) - S \right\rVert_{L^2(\Omega;\on{HS})}$ is infinite. A simple sufficient condition for finiteness of the latter norm is that $F$ has finite chaos decomposition (see Section~\ref{sec:fourth-moment-bounds}). In the case of an infinite decomposition, some control on the tail is needed.
\item In principle, Theorem~\ref{thm:1} can also be used to prove weak convergence in a Banach space setting. Starting from a Gaussian random variable on a separable Banach space $B$, it is always possible (see~\cite[Lemma 2.1]{kuelbs:1970:gaussian-measures-banach}) to densely embed $B$ in a separable Hilbert space $K$ such that the Borel sets of $B$ are generated by the inner product of $K$. Then, by applying our methods, one obtains weak convergence in $K$, which in turn implies weak convergence in $B$.
  \end{enumerate}
\end{remark}

\subsection{Fourth Moment bounds via chaos expansions}
\label{sec:fourth-moment-bounds}

In this section, we show how the carré du champ bounds obtained in Theorem
\ref{thm:1} can be further estimated by the first four moments of the approximating random variable or sequence. For this, we need to assume that the generator satisfies the following, generalized version of an abstract polynomial chaos property first stated in~\cite{azmoodeh-campese-poly:2014:fourth-moment-theorems} for the finite-dimensional case.

\begin{definition}
  \label{def:2}
  Denote by $\widetilde{L}$ the one-dimensional counterpart of $L$ as introduced in Section~\ref{sec:dirichl-struct-hilb} and recall that $L$ and $\widetilde{L}$ have the same spectrum. Let $\lambda$, $\eta$ be two of their common eigenvalues. Two eigenvectors $F \in \on{ker} \left( L + \lambda \on{Id} \right)$ and $G \in \on{ker} \left( L + \eta \on{Id} \right)$ are called \emph{jointly chaotic}, if
\begin{equation*}
  \left\langle F,G \right\rangle_K \in \bigoplus_{\substack{\alpha \in \Lambda \\ \alpha \leq \lambda + \eta}} \on{ker} \left( \widetilde{L} + \alpha \on{Id} \right),
\end{equation*}
where $\Lambda$ denotes the spectrum of $L$.
An eigenvector $F \in \on{ker}\left( L + \lambda \on{Id} \right)$ is called \emph{chaotic} if it is jointly chaotic with itself, i.e., if
\begin{equation*}
  \left\lVert F \right\rVert_K^2 \in \bigoplus_{\substack{\alpha \in \Lambda \\ \alpha \leq 2\lambda}} \ker \left( \widetilde{L} + \alpha \on{Id} \right).
\end{equation*}
The generator $L$ is called chaotic, if any two of its eigenfunctions are jointly chaotic.
\end{definition}

Prime examples of chaotic generators are those whose eigenspaces
consist of (closures of) multivariate polynomials, such as the
Hilbert-valued Ornstein-Uhlenbeck generator, Laguerre or Jacobi
generators, in finite or infinite dimension. The Ornstein-Uhlenbeck case will be covered in depth in Section~\ref{sec:hilb-valu-wien}, precise definitions for the other two generators can for example be found in~\cite{azmoodeh-campese-poly:2014:fourth-moment-theorems}.

We will also make use of the following covariance condition.

\begin{definition}
A random variable $F \in L^2(\Omega;K)$ is said to satisfy the
\emph{covariance condition} if it holds that
\begin{equation}
 \label{eq:18}
2\on{Cov} \left(
      \left\langle F,u \right\rangle
   ,
   \left\langle F,v \right\rangle
 \right)^2
 \leq
\on{Cov}
 \left(
   \left\langle F,u \right\rangle^2
   ,
   \left\langle F,v \right\rangle^2
 \right)
\end{equation}
for any two orthonormal vectors $u,v \in K$.
\end{definition}

It will be proved later that both the covariance condition and the chaotic property is satisfied whenever $F$ is an eigenfunction of the Ornstein-Uhlenbeck generator.

Now we can state the main result of this section.

\begin{theorem}
  \label{thm:14}
  Let $F \in \mathbb{D}$ with chaos expansion $F = \sum_{p=1}^\infty F_p $, where $LF_p=-\lambda_p F_p$ and assume that $L$ is chaotic and its eigenfunctions verify the covariance assumption~\eqref{eq:18}. Denote the covariance operators of $F_p$ by $S_p$, so that $F$ has covariance operator $S=\sum_{p=1}^{\infty} S_p$. Then
\begin{equation}
  \label{eq:61}
  \left\lVert \Gamma(F,-L^{-1}F) - S \right\rVert_{L^2(\Omega;\on{HS}(K))}
  \leq
  \sqrt{M(F) + C(F)},
\end{equation}
where
\begin{align}
  \label{eq:73}
  M(F)
  &=
    \frac{1}{\sqrt{3}}
       \sum_{p,q=1}^{\infty}
       c_{p,q}
\sqrt{
    \E \left( \left\lVert F_p \right\rVert^4 \right)
    \left(
           \E \left( \left\lVert F_q \right\rVert^4 \right)
     -
     \E \left( \left\lVert F_q \right\rVert^2 \right)^2
     -
     2
    \left\lVert S_{q} \right\rVert_{\on{HS}}^2
     \right)
    },
  \\
 \label{eq:74}
  C(F)
  &=
    \sum_{\substack{p,q \in \N\\p \neq q}}^{}
  a_{p,q}
  \on{Cov}
  \left(
  \left\lVert F_p \right\rVert^2
  ,
   \left\lVert F_q \right\rVert^2
      \right),
\end{align}
and the constants $a_{p,q}$ and $c_{p,q}$ are given by $a_{p,q} = (\lambda_p+\lambda_q)/2\lambda_q$ and
\begin{equation*}
c_{p,q} =
  \begin{cases}
    1+\sqrt{3}    & \qquad \text{if $p=q$,}
    \\
    a_{p,q} & \qquad \text{if $p \neq q$},
  \end{cases}
\end{equation*}
respectively.
\end{theorem}

Before proving Theorem~\ref{thm:14}, let us give the following restatement of $M$ in terms of fourth moments only.

\begin{proposition}
  \label{prop:3}
  In the setting of Theorem~\ref{thm:14}, it holds that
  \begin{equation}
 \label{eq:69}
 M(F) =
 \frac{1}{\sqrt{3}}
 \sum_{p,q=1}^{\infty} c_{p,q}
 \sqrt{\E \left( \left\lVert F_p \right\rVert^4 \right)
   \left( \E \left( \left\lVert F_q \right\rVert^4 \right)
     -
     \E \left( \left\lVert Z_q \right\rVert^4 \right)
   \right)
},
\end{equation}
where the $Z_p$ are centered Gaussian random variables with the same covariance operators as the $F_p$.
\end{proposition}

\begin{proof}
Using similar arguments as in~\cite[Proof of Theorem 4.2]{nourdin-rosinski:2014:asymptotic-independence-multiple} combined with ~\cite[Theorem 2.1]{kannan-kannappan:1975:characterization-gaussian-measures}, a straightforward calculation yields
\begin{equation}
  \label{eq:22}
  \E \left( \left\lVert Z_p \right\rVert^4 \right)
  =
  \E \left( \left\lVert F_p \right\rVert^2 \right)^2
  +
  2
  \left\lVert S_{p} \right\rVert_{\on{HS}}^2.
\end{equation}
\end{proof}

\begin{proof}[Proof of Theorem~\ref{thm:14}]
  The idea of the proof is to transfer the Dirichlet structure from
  $L^2(\Omega;K)$ to $L^2(\Omega;\mathbb{R})$ by expanding in an
  orthonormal basis and working on the coefficients, afterwards
  reassembling everything again. To this end, let $\left\{ e_i \colon
    i \in \N \right\}$ be an orthonormal basis of $K$ and denote $F_i
= \left\langle F,e_i \right\rangle$, as well as $F_{p,i} =
\left\langle F_p,e_i \right\rangle$ for $i \in \N$. Note that
\begin{equation*}
  \left\langle
    \Gamma(F,-L^{-1}F)e_i,
    e_j
  \right\rangle
  =
  \widetilde{\Gamma}(F_i,-\widetilde{L}^{-1}F_j),
\end{equation*}
where $\widetilde{\Gamma}$ and $\widetilde{L}$ are the real-valued counterparts of $\Gamma$ and $L$ (see Section~\ref{sec:dirichl-struct-hilb}). To improve readability, we will not make any notational distinction between the real-valued and Hilbert-valued case and therefore denote $\widetilde{\Gamma}$ and $\widetilde{L}$ by the symbols $\Gamma$ and $L$ as well throughout the proof. The meaning can always unambigously be inferred from the context, depending on whether the arguments are $K$- or $\mathbb{R}$-valued.

Define the cross-covariance operators $C_{p,q} \colon K \to K$ via the identity
\begin{equation*}
  \E \left( \left\langle F_p,k \right\rangle \left\langle F_q,l \right\rangle \right) = \left\langle C_{p,q}k,l \right\rangle, \qquad k,l \in K.
\end{equation*}
Then, $C_{p,p}=S_p$ and, by orthogonality, $C_{p,q}=0$ if $p \neq q$. Therefore,
\begin{equation*}
  S = \sum_{p=1}^{\infty} S_p = \sum_{p,q=1}^{\infty} C_{p,q},
\end{equation*}
and consequently
\begin{align}
  \notag
  \left\lVert \Gamma(F,-L^{-1}F) - S \right\rVert_{L^2(\Omega;\on{HS}(K))}
  &\leq
    \sum_{p,q=1}^{\infty}
    \left\lVert \Gamma(F_p,-L^{-1}F_q) - C_{p,q} \right\rVert_{L^2(\Omega;\on{HS}(K))}
 \\ &\notag =
    \sum_{p,q=1}^{\infty}
    \sqrt{
    \sum_{i,j=1}^{\infty}
    \E \left(
    \left( \Gamma(F_{p,i},-L^{-1}F_{q,j}) - \E \left( F_{p,i} F_{q,j} \right) \right)^2
    \right)
    }
\\ &= \label{eq:71}
     \sum_{p,q=1}^{\infty}
     \sqrt{
     \sum_{i,j=1}^{\infty}
     \on{Var}
       \left(
       \Gamma(F_{p,i},-L^{-1}F_{q,j})
       \right)
     },
\end{align}
Note that all carré du champ operators appearing in the double sum~\eqref{eq:71} are acting on real valued random variables, so that known results from the finite-dimensional theory can be applied.

For $p=q$, Theorem 3.2 in~\cite{azmoodeh-campese-poly:2014:fourth-moment-theorems} yields
\begin{equation}
\on{Var}
       \left(
\Gamma(F_{q,j},-L^{-1}F_{q,j})
  \right)
  \leq
    \frac{1}{3}
    \left(
    \E \left( F_{q,j}^4 \right)
    -
    3
    \E \left( F_{q,j}^2 \right)^2
    \right),
\end{equation}
which, together with the covariance condition~\eqref{eq:18}, implies that
\begin{align}
  \notag
         \sum_{j=1}^{\infty}
       \on{Var}&
       \left(
\Gamma(F_{q,j},-L^{-1}F_{q,j})
  \right)
  \leq
    \frac{1}{3}
    \sum_{j=1}^{\infty}
    \left(
    \E \left( F_{q,j}^4 \right)
    -
    3
    \E \left( F_{q,j}^2 \right)^2
    \right)
  \\ &\leq \notag
    \frac{1}{3}
\sum_{i,j=1}^{\infty}
    \left(
    \E \left( F_{q,i}^2 F_{q,j}^2 \right)
  -
  \E \left( F_{q,i}^2 \right)
  \E \left( F_{q,j}^2 \right)
  -
  2
  \E \left( F_{q,i} F_{q,j} \right)^2
       \right)
  \\ &= \label{eq:72}
       \frac{1}{3}
       \left(
       \E \left(
       \left\lVert F_q \right\rVert^4
       \right)
       -
       \E \left(
       \left\lVert F_q \right\rVert^2
       \right)^2
       -
       2
       \left\lVert S_q \right\rVert_{\on{HS}}^2
       \right).
\end{align}

For $p \neq q$, similar calculations as in~\cite[Proof of Theorem 1.2]{campese-nourdin-peccati-ea:2016:multivariate-gaussian-approximations} (which in turn relied on the main ideas of~\cite{azmoodeh-campese-poly:2014:fourth-moment-theorems}) lead to
\begin{align*}
       \on{Var}
       \left(
       \Gamma(F_{p,i},-L^{-1}F_{q,j})
  \right)
  &\leq
  a_{p,q}
  \Big( 
       \E \left( F_{p,i}^2F_{q,j}^2 \right)
       -
       \E \left( F_{p,i}^2 \right) \E \left( F_{q,j}^2 \right)
       -
       2      \E \left( F_{p,i}F_{q,j} \right)^2
\\ &\qquad \qquad \qquad \qquad 
         -  \E \left(
       F_{p,i}^2
       \left(
       \Gamma(F_{q,j},-L^{-1}F_{q,j})
       -
       \E \left( F_{q,j}^2 \right)
       \right)
    \right)
  \Big),
\end{align*}
so that
\begin{align}
  \notag
  \sum_{i,j=1}^{\infty}
         &\on{Var}
       \left(
       \Gamma(F_{p,i},-L^{-1}F_{q,j})
  \right)
  \\ & \notag
  \leq
  a_{p,q}
  \left(
    \E \left(
      \left\lVert F_p \right\rVert^2
      \left\lVert F_q \right\rVert^2
      \right)
       -
       \E \left(
         \left\lVert F_p \right\rVert^2
       \right)
       \E \left(
         \left\lVert F_q \right\rVert^2
         \right)
       -
       2  \sum_{i,j=1}^{\infty} \E \left( F_{p,i}F_{q,j} \right)^2
       \right.
       \\ &\left.  \qquad \qquad - \notag
                  \sum_{j=1}^{\infty}
           \E \left(
       \left\lVert F_p \right\rVert^2
       \left(
\Gamma(F_{q,j},-L^{-1}F_{q,j})
       -
       \E \left( F_{q,j}^2 \right)
       \right)
    \right)
            \right)
  \\ &\leq \notag
         a_{p,q}
  \left(
    \E \left(
      \left\lVert F_p \right\rVert^2
      \left\lVert F_q \right\rVert^2
      \right)
       -
       \E \left(
         \left\lVert F_p \right\rVert^2
       \right)
       \E \left(
         \left\lVert F_q \right\rVert^2
         \right)
       -
       2  \sum_{i,j=1}^{\infty} \E \left( F_{p,i}F_{q,j} \right)^2
       \right.
  \\ &\left. \qquad \qquad + \notag
       \sqrt{\E \left(
       \left\lVert F_p \right\rVert^4
       \right)}
       \sqrt{
       \sum_{j=1}^{\infty}
       \on{Var}
       \left(
\Gamma(F_{q,j},-L^{-1}F_{q,j})
       \right)
       }
       \right).
\end{align}
Together with~\eqref{eq:72}, we thus get for $p=q$ that
\begin{align*}
\sum_{i,j=1}^{\infty}
         &\on{Var}
       \left(
       \Gamma(F_{p,i},-L^{-1}F_{p,j})
     \right) \notag
     \\ &\leq
     \E \left( \left\lVert F_p \right\rVert^4 \right)
     -
     \E \left( \left\lVert F_p \right\rVert^2 \right)^2
     -
     2
     \left\lVert S_p \right\rVert_{\on{HS}}^2 \notag
    \\ & \qquad \qquad +
    \sqrt{
      \frac{1}{3}
      \E \left( \left\lVert F_p \right\rVert^4 \right)
    }
    \sqrt{
           \E \left( \left\lVert F_p \right\rVert^4 \right)
     -
     \E \left( \left\lVert F_p \right\rVert^2 \right)^2
     -
     2
         \left\lVert S_p \right\rVert_{\on{HS}}^2} \notag
  \\ &\leq
       \frac{1+\sqrt{3}}{\sqrt{3}}
    \sqrt{
      \E \left( \left\lVert F_p \right\rVert^4 \right)
    }
    \sqrt{
           \E \left( \left\lVert F_p \right\rVert^4 \right)
     -
     \E \left( \left\lVert F_p \right\rVert^2 \right)^2
     -
     2
       \left\lVert S_p \right\rVert_{\on{HS}}^2
       }
\end{align*}
and for $p \neq q$ that
\begin{align*}
  \notag
\sum_{i,j=1}^{\infty}
         &\on{Var}
       \left(
       \Gamma(F_{p,i},-L^{-1}F_{q,j})
           \right)
  \\ &\leq
         a_{p,q}
  \left(
    \E \left(
      \left\lVert F_p \right\rVert^2
      \left\lVert F_q \right\rVert^2
      \right)
       -
       \E \left(
         \left\lVert F_p \right\rVert^2
       \right)
       \E \left(
         \left\lVert F_q \right\rVert^2
         \right)
\right)
  \\ & \qquad \qquad + \notag
\frac{a_{p,q}}{\sqrt{3}}
       \sqrt{
       \E \left( \left\lVert F_p \right\rVert^4 \right)
       }
       \sqrt{
           \E \left( \left\lVert F_q \right\rVert^4 \right)
     -
     \E \left( \left\lVert F_q \right\rVert^2 \right)^2
     -
     2
         \left\lVert S_q \right\rVert_{\on{HS}}^2
       },
\end{align*}
from which the asserted bound follows.

\end{proof}

Inspecting the proof of Theorem~\ref{thm:14}, it becomes apparent that
for the case where $F=F_p$ is a chaotic eigenfunction, we can remove one square root. In other words, the following holds.

\begin{corollary}
  \label{thm:17}
  If $F$ is a chaotic eigenfunction of $L$ and $Z$ a centered, non-degenerate Gaussian on $K$, both having covariance operator $S$, then
\begin{multline}
 \label{eq:76}
    \left\lVert \Gamma(F,-L^{-1}F) - S \right\rVert_{L^2(\Omega;\on{HS}(K))}
  \\  \leq
    \frac{1+\sqrt{3}}{\sqrt{3}}
    \sqrt{
      \E \left( \left\lVert F \right\rVert^4 \right)
    \left(
           \E \left( \left\lVert F \right\rVert^4 \right)
     -
     \E \left( \left\lVert F \right\rVert^2 \right)^2
     -
     2
    \left\lVert S\right\rVert_{\on{HS}}^2
     \right)
     }
\\    =
\frac{1+\sqrt{3}}{\sqrt{3}}
    \sqrt{
      \E \left( \left\lVert F \right\rVert^4 \right)
    \left(
           \E \left( \left\lVert F \right\rVert^4 \right)
     -
           \E \left( \left\lVert Z \right\rVert^4 \right)
     \right)
     }.
  \end{multline}
\end{corollary}

Combining Theorems~\ref{thm:14} and~\ref{thm:1}, the following moment bound is obtained.

\begin{theorem}
  \label{thm:6}
Let $Z$ be a centered Gaussian, non-degenerate random variable on $K$, assume that $L$ is chaotic and let $F \in L^2(\Omega;K)$ with chaos expansion $F = \sum_{p=1}^\infty F_p $, where
$LF_p=-\lambda_p F_p$. Denote the covariance operators of $Z$, $F$ and $F_p$ by $S$ and $T$ and $S_p$, respectively. Then the following two statements are true.

\begin{enumerate}[(i)]
 \item If $F_p$ satisfies the covariance condition~(\ref{eq:18}) for all $p \in \N$, then
   \begin{equation}
     \label{eq:66}
    d_2(F,Z)
  \leq
       \frac{1}{2}
       \left( \sqrt{M(F) + C(F)} + \left\lVert S-T \right\rVert_{\on{HS}} \right)
\end{equation}
where the quantities $M(F)$ and $C(F)$ are given by~\eqref{eq:73} (or equivalently~\eqref{eq:69}) and~\eqref{eq:74}, respectively.
\item
If $F=F_p$ for some eigenfunction $F_p \in \on{ker} \left( L + \lambda_p \on{Id} \right)$, then
\begin{multline*}
d_2(F,Z)
  \leq
\frac{1+\sqrt{3}}{2 \sqrt{3}}
       \sqrt{
       \E \left( \left\lVert F \right\rVert^4 \right)
       \left(
\E \left( \left\lVert F \right\rVert^4 \right)
     -
       \E \left( \left\lVert F \right\rVert^2 \right)^2
       -
     2
       \left\lVert S_{p} \right\rVert_{\on{HS}}^2
       \right)}
 \\
      +
      \frac{1}{2}
  \left\lVert S - S_p \right\rVert_{\on{HS}}.
    \end{multline*}
\end{enumerate}
\end{theorem}

\begin{remark}
 \label{rmk:5}
 In applications, given $F=\sum_{p=1}^{\infty} F_p$, it might sometimes be favorable to apply Theorem~\ref{thm:6} to the truncated series $G_N = \sum_{p=1}^N F_p$ via the simple estimate
\begin{equation*}
  d_2(F,Z) \leq d_2(G_N,Z) + d_2(G_N,F),
\end{equation*}
so that the expressions $M(G_N)$ and $C(G_N)$ are no longer infinite series but finite sums. To handle the additional term $d_2(G_{N},F)$, one then needs control on the tails $\E \left( \left\lVert F-G_N \right\rVert \right)$, for example via
\begin{equation*}
  \E \left( \left\lVert F-G_N \right\rVert \right)^2
  \leq \E \left( \left\lVert F-G_{N} \right\rVert^2 \right)
  = \sum_{p=N+1}^{\infty} \on{tr} S_p.
\end{equation*}
\end{remark}

Of course, in the setting of Theorem~\ref{thm:6}, if $K$ is assumed to
have finite dimension $d$, then the right hand side of~\eqref{eq:66}
also bounds the Wasserstein distance $d_W(F,Z)$ (with constant $1/2$ replaced by $C_{s,d}$ of Theorem~\ref{thm:1}). Let us now state two central limit theorems which are direct consequences of Theorem~\ref{thm:6}. The first one is an abstract Fourth Moment Theorem.

\begin{theorem}[Abstract fourth moment theorem]
  \label{thm:11}
  Let $Z$ be a centered, non-degenerate Gaussian random variable on
  $K$ and $\left\{ F_n \colon n \in \N \right\}$ be a sequence of $K$-valued chaotic
  eigenfunctions such that $\E \left( \left\lVert F_n \right\rVert^2 \right) \to \E \left( \left\lVert Z \right\rVert^2 \right)$. Consider the following two asymptotic relations, as $n \to \infty$:
  \begin{enumerate}[(i)]
  \item $F_n$ converges in distribution to $Z$;
\item\label{item:6} $\E \left( \left\lVert F_n \right\rVert^4 \right)\to \E \left( \left\lVert Z \right\rVert^4 \right)$.
\end{enumerate}
Then, $(ii)$ implies $(i)$, and the converse implication holds
whenever the moment sequence $\left\{ \norm{F_n}^4 \colon n\geq
  1\right\}$ is uniformly integrable.
\end{theorem}
\begin{proof}
  Denote the covariance operators of $Z$ and the $F_n$ by $S$ and $S_n$, respectively. Then by assumption $\on{tr}(S_n-S) \to 0$.
The fact that $(ii)$ implies $(i)$ is a direct consequence of Theorem
\ref{thm:6}. The converse
implication follows immediately if the additional uniform
integrability condition is assumed to hold.
\end{proof}

\begin{remark}
  \begin{enumerate}[(i)]
  \item As is well known, a sufficient condition for uniform integrability
of the sequence $\left\{ \norm{F_n}^4 \colon n\geq
  1\right\}$ is given by $\sup_{n \geq 1}
\Ex{\norm{F_n}^{4+\varepsilon}} <\infty$ for some $\varepsilon >0$.
  \item Theorem~\ref{thm:11} is a Hilbert-valued generalization of the Gaussian Fourth Moment Theorems derived in~\cite{azmoodeh-campese-poly:2014:fourth-moment-theorems} ($K=\mathbb{R}$) and~\cite{campese-nourdin-peccati-ea:2016:multivariate-gaussian-approximations} ($K=\mathbb{R}^d$ with Euclidean inner product). As further special cases, taking $L$ to be the Ornstein-Uhlenbeck generator on $L^2(\Omega,\R)$, the classical Fourth Moment Theorem of~\cite{nualart-peccati:2005:central-limit-theorems} ($K=\mathbb{R}$) and Theorem 4.2  of~\cite{nourdin-rosinski:2014:asymptotic-independence-multiple} ($K=\R^d$ with Euclidean inner product) are included. Further details on these latter two cases will be provided in Section~\ref{sec:fourth-moment-contr}.
\end{enumerate}
\end{remark}

For functionals with infinite chaos expansions, the corresponding limit theorem reads as follows. Again, the proof is a straightforward application of Theorem~\ref{thm:6}.

\begin{theorem}
  \label{thm:12}
Let $Z$ be a centered, non-degenerate Gaussian random variable on $K$ with covariance
  operator $S$ and let $\left\{ F_n \colon n\in \N\right\}$ be a sequence of square integrable, $K$-valued
random variables with chaos decomposition
\begin{equation*}
F_n = \sum_{p=1}^{\infty}F_{p,n},
\end{equation*}
where, for each $n,p \geq 1$, $F_{p,n}$ is a chaotic eigenfunction associated
to the eigenvalue $-\lambda_{p}$ (of the operator $-L$) and verifying the covariance condition~\eqref{eq:18}. For $n,p \in \N$, let $S_n$ and $S_{p,n}$ be the covariance operators of $F_n$ and $F_{p,n}$, respectively. Suppose that:
\begin{enumerate}[(i)]
\item \label{item:8} There exists a sequence $\left\{ S_p \colon p\in \N\right\}$ of covariance operators such that $S=\sum_{p=1}^{\infty} S_p$ is of trace class, $\on{tr}(S_{p,n} - S_p) \to 0$ as $n \to \infty$
and
\begin{equation}
  \label{eq:14}
  \lim_{N \to \infty}
  \sup_{n \in \N} \sum_{p=N}^{\infty} \on{tr}(S_{p,n}) = 0
\end{equation}
\item \label{item:7} For all $p,q \in \N$, it holds that
\begin{equation*}
  \E \left(
    \left\lVert F_{p,n} \right\rVert^4
  \right)
  -
  \E \left(
    \left\lVert F_{p,n} \right\rVert^2
  \right)^2
  -
  2 \left\lVert S_{{p,n}} \right\rVert_{\on{HS}}^2
  \to 0
\end{equation*}
and, if $p \neq q$,
\begin{equation*}
  \E \left( \left\lVert F_{p,n} \right\rVert^2 \left\lVert F_{q,n} \right\rVert^2 \right)
  -
  \E \left( \left\lVert F_{p,n} \right\rVert^2 \right)
  \E \left(
    \left\lVert F_{q,n} \right\rVert^2
  \right)
  \to 0
\end{equation*}
as $n \to \infty$.
\end{enumerate}
Then $F_n$ converges in distribution to $Z$ as $n \to \infty$.
\end{theorem}
\begin{proof}
  For $N \in \N$, define $F_{n,N} = \sum_{p=1}^N F_{p,n}$, $R_{n,N} = F_n-F_{n,N} = \sum_{p=N+1}^{\infty} F_{p,n}$ and let $Z_N$ be a centered Gaussian random variable on $K$ with covariance operator $\sum_{p=1}^N S_p$. Now let $\varepsilon > 0$ and note that
  \begin{equation}
    \label{eq:35}
  d_2(F_n,Z) \leq
  d_2(F_n,F_{n,N})
  +
  d_2(F_{n,N},Z_N)
  +
  d_2(Z_N,Z).
\end{equation}
For $h \in C_b^2(K)$, we get by Lipschitz-continuity that
\begin{equation*}
 d_2(F_n,F_{n,N}) \leq \E \left( \left\lVert R_{n,N} \right\rVert \right)
 \leq \sqrt{
   \E \left( \left( \left\lVert R_{n,N} \right\rVert \right)^2 \right)
 }
 \leq
 \sqrt{
   \sum_{p=N+1}^{\infty} \on{tr}(S_{p,n})
 }.
\end{equation*}
Similarly,
\begin{equation*}
  d_2(Z_N,Z) \leq \sqrt{
    \sum_{p=N+1}^{\infty} \on{tr}(S_p)
    }
\end{equation*}
The above two calculations, together with assumption~\emph{(\ref{item:8})}, yield the existence of $N \in \mathbb{N}$, not dependent of $n$, such that
\begin{equation}
 \label{eq:33}
  d_2(F_{n},F_{n,N}) +  d_2(Z_N,Z) < \varepsilon.
\end{equation}
By assumption~\emph{(\ref{item:7})} and Theorem~\ref{thm:6}, we also have that
\begin{equation*}
  d_2(F_{n,N},Z_{N}) \to 0
\end{equation*}
as $n \to \infty$, so that in view of~\eqref{eq:35},
\begin{equation*}
  0 \leq \lim_{n \to \infty} d(F_n,Z) < \varepsilon.
\end{equation*}
The assertion follows as $\varepsilon$ was arbitrary.
\end{proof}
Although we stated Theorems~\ref{thm:11} and~\ref{thm:12} in a qualitative way, it should be clear that the convergences in both results are actually quantified by Theorem~\ref{thm:6}.

\section{Hilbert-valued Wiener structures}
\label{sec:hilb-valu-wien}

In this section, we apply our general results to the special Dirichlet structure induced by the Ornstein-Uhlenbeck generator. This leads to Hilbert-valued Wiener chaos and a carré du champ operator given by a gradient of Hilbert-valued Malliavin derivatives. The eigenfunctions are multiple Wiener-Itô integrals with Hilbert-valued deterministic kernels.
This additional structure allows to express the moment bounds of the previous sections in terms of kernel contractions, which in the finite-dimensional case have already proved themselves to be very useful in applications, due to their comparatively easy computability when compared to moments.

\subsection{Malliavin calculus}
\label{sec:malliavin-calculus}

We present some basic elements of the Malliavin calculus on Hilbert
spaces, as introduced
in~\cite{leon-nualart:1998:stochastic-evolution-equations} (based on
the earlier
work~\cite{grorud-pardoux:1992:integrales-hilbertiennes-anticipantes}). Expository
references are~\cite[Chapter
4]{kruse:2014:strong-weak-approximation},~\cite[Chapter 5]{carmona-tehranchi:2006:interest-rate-models}.
The authoritative reference for Malliavin calculus on $L^2(\Omega;\R)$
is~\cite{nualart:2006:malliavin-calculus-related}, a further excellent
exposition can be found in~\cite[Chapters 1 and 2]{nourdin-peccati:2012:normal-approximations-malliavin}.

\subsubsection{The Malliavin derivative and divergence operators}

Let $\left\{ W(h) \colon h \in \mathfrak{H} \right\}$ be an isonormal Gaussian
process with underlying separable Hilbert space $\mathfrak{H}$, that is $\left\{ W(h)
  \colon h \in \mathfrak{H} \right\}$ is a centered family of Gaussian random
variables, defined on a complete probability space $(\Omega,\mathcal{F},P)$, satisfying
\begin{equation*}
\Ex{W(h_1)W(h_2)} = \left\langle h_1, h_2 \right\rangle_{\mathfrak{H}}, \qquad
h_1, h_2 \in \mathfrak{H}.
\end{equation*}
We assume that the $\sigma$-algebra $\mathcal{F}$ is generated by $W$.
Let $K$ be another separable Hilbert space and denote by
$\mathcal{S}\otimes K$ the class of smooth $K$-valued random variables
$F \colon \Omega \to K$ of the form $F = f \left( W(h_1), \ldots, W(h_n) \right) \otimes v$,
where $f \in C_b^{\infty} \left( \R^n \right)$, $h_1, \ldots, h_n \in
\mathfrak{H}$, $v \in K$, and linear combinations thereof. $\mathcal{S}\otimes K$
is dense in $L^2(\Omega;K)$ and for $F \in
\mathcal{S}\otimes K$, define the Malliavin derivative $DF$ of $F$ as
the $\mathfrak{H} \otimes K$-valued random variable  given by
\begin{equation}
\label{eq:1}
DF = \sum_{i=1}^n\partial_i f \left( W(h_1), \ldots, W(h_n) \right)h_i
\otimes v.
\end{equation}
It can be shown that
$D$ is a closable operator from $L^2 \left( \Omega;K \right)$ into
$L^2 \left( \Omega; \mathfrak{H} \otimes K \right)$, and from now we continue to use the symbol $D$ to denote the closure. The domain of $D$, denoted by $\mathbb{D}^{1,2}(K)$, is the closure of $\mathcal{S}\otimes K$
with respect to the Sobolev norm $\norm{F}_{\mathbb{D}^{1,2}(K)}^2 = \norm{F}_{L^2(\Omega;K)}^2 +
\norm{DF}_{L^2(\Omega;\mathfrak{H} \otimes K)}^2$. Similarly, for $k \geq 2$, let $\mathbb{D}^{k,2}(K)$ denote the
closure of $\mathcal{S}\otimes K$ with respect to the Sobolev norm $\norm{F}_{\mathbb{D}^{k,2}(K)}^2 = \norm{F}_{L^2(\Omega;K)}^2 +
\sum_{i=1}^k\norm{D^iF}_{L^2(\Omega;\mathfrak{H}^{\otimes i}\otimes
  K)}^2$. For any $k \geq 2$, the operator $D^k$ can be interpreted as the
iteration of the Malliavin derivative operator defined in
\eqref{eq:1}. As $D$ is a closed linear operator from $\mathbb{D}^{1,2}(K)$ to
$L^2(\Omega;\mathfrak{H}\otimes K)$, it has an adjoint operator, denoted by
$\delta$, which maps a subspace of $L^2(\Omega;\mathfrak{H}\otimes K)$ into
$L^2(\Omega;K)$ through the duality relation
\begin{equation*}
\Ex{\left\langle DF,\eta \right\rangle_{\mathfrak{H}\otimes K}} =
\Ex{\left\langle F, \delta(\eta) \right\rangle_K},
\end{equation*}
for any $F \in \mathbb{D}^{1,2}(K)$ and $\eta \in \on{dom}(\delta)$. The domain of $\delta$, denoted by $\on{dom}(\delta)$, is the
subset of random variables $\eta \in L^2(\Omega;\mathfrak{H} \otimes K)$ such
that $\abs{\Ex{\left\langle DF,\eta \right\rangle_{\mathfrak{H} \otimes K}}} \leq
C_{\eta}\norm{F}_{L^2(\Omega;K)}$, for all $F \in \mathbb{D}^{1,2}(K)$, where $C_{\eta}$ is a positive
constant depending only on $\eta$. Since $D$ is a form of gradient,
its adjoint $\delta$ should be interpreted as a divergence, so that it
is referred to as the divergence operator. Similarly, for any $k \geq
2$, we denote by
$\delta^k$ the adjoint of $D^k$ as an operator from
$L^2(\Omega;\mathfrak{H}^{\otimes k}\otimes K)$ to $L^2(\Omega; K)$
with domain $\on{dom}(\delta^k)$.

\subsubsection{Multiple integrals and chaos decomposition}
Any $K$-valued random variable $F \in L^2(\Omega;K)$ can be decomposed
as
\begin{equation}
  \label{eq:2}
F = \sum_{n=0}^{\infty}\delta^n(f_n),
\end{equation}
where the kernel $f_n \in \mathfrak{H}^{\odot n}\otimes K$ are uniquely
determined by $F$, where $\mathfrak{H}^{\odot n}$ denotes the $n$-fold symmetrized
tensor product of $\mathfrak{H}$. The representation \eqref{eq:2} is called the
chaos decomposition of $F$, and for each $n \geq 0$, $\delta^n(f_n)$
is an element of the closure of $\mathfrak{H}_n \otimes K$ with respect
to the norm on $L^2(\Omega;K)$, where the so-called $n$-th Wiener
chaos $\mathfrak{H}_n$ is defined to be closed linear
subspace of $L^2(\Omega)$ generated by the random variables $\left\{
  H_n(W(h))\colon h \in \mathfrak{H},\ \norm{h}_{\mathfrak{H}}=1 \right\}$, where $H_n$
is the $n$-th Hermite polynomial given by $H_n(x) =  (-1)^{n} \me^{x^2/2} \left( \frac{d}{dx} \right)^n
\me^{-x^2/2}$ (recall that $\mathfrak{H}_0$ is identified with $\R$). For any $n \geq 0$, the $K$-valued random
variable $\delta^n(f_n)$ is usually denoted by $I_n(f_n)$ and called the ($K$-valued) multiple
Wiener integral of order $n$ of $f_n$. In the particular case where $K=\R$, these integrals coincide with the ones defined in~\cite{nualart:2006:malliavin-calculus-related}. Denote by $J_n$ the linear
operator on $L^2 (\Omega)$ given by the orthogonal projection onto
$\mathfrak{H}_n$, and by $J_n^K$ the
extension of $J_n \otimes \on{Id}_K$ to $L^2(\Omega;K)$. Then, it holds that
$J_n^K F = I_p(f_n)$. Let $\left\{ e_k \colon k \geq 0 \right\}$ be an
orthonormal basis of $\mathfrak{H}$. Given $f \in \mathfrak{H}^{\odot n}$ and $g
\in \mathfrak{H}^{\odot m}$, for every $r = 0,\ldots, n \wedge m$, the
$r$-th contraction of $f$ and $g$ is the element of $\mathfrak{H}^{\otimes
  (n+m-2r)}$ defined as
\begin{equation}
  \label{eq:55}
f \otimes_r g = \sum_{i_1, \ldots, i_r =0}^{\infty} \left\langle
  f,e_{i_1}\otimes \cdots \otimes e_{i_r} \right\rangle_{\mathfrak{H}^{\otimes r}}
\otimes \left\langle
  g,e_{i_1}\otimes \cdots \otimes e_{i_r} \right\rangle_{\mathfrak{H}^{\otimes r}}.
\end{equation}
We denote by $f \widetilde{\otimes}_r g$ the symmetrization (average
over all permutations of the arguments) of $f \otimes_r g$.
Given an orthonormal basis $\left\{ v_k \colon k \geq 0 \right\}$ of $K$,
the following multiplication formula is satisfied by $K$-valued multiple Wiener
integrals: for two arbitrary basis elements $v_i, v_j$ of $K$ and for
$f \in \mathfrak{H}^{\odot n} \otimes K$ and $g \in \mathfrak{H}^{\odot m} \otimes K$,
define $f_i = \left\langle f,v_i \right\rangle_K$ and $g_j =
\left\langle g,v_j \right\rangle_K$. Then
\begin{equation}
  \label{multiplicationformulamultint}
 I_n(f_i) I_m(g_j) = \sum_{r=0}^{n \wedge m} r! \binom{n}{r}\binom{m}{r}
  I_{n+m-2r}(f_i \widetilde{\otimes}_{r} g_j).
\end{equation}
Finally, the action of the Malliavin derivative operator on a $K$-valued
multiple Wiener integral of the form $I_n(f) \in L^2(\Omega;K)$, where $f \in \mathfrak{H}^{\odot
  n}\otimes K$, is given by $DI_n(f)= n I_{n-1}(f(\cdot)) \in L^2(\Omega;\mathfrak{H}\otimes K)$.

\subsection{Fourth moment and contraction bounds}
\label{sec:fourth-moment-contr}

In this section, we are going to apply our abstract results to the Dirichlet structure given by the Ornstein-Uhlenbeck generator, acting on $L^2(\Omega;K)$, where $K$ is a real, separable Hilbert space and the $\sigma$-algebra of the underlying probability space is generated by an isonormal Gaussian process $W$, indexed by a real, separable Hilbert space $\mathfrak{H}$. The Ornstein-Uhlenbeck generator, commonly denoted by $-L$ in this context, is then defined as $-L=\delta D$. Its spectrum is given by the non-negative integers and the eigenspace asociated to the eigenvalue $p  \in \mathbb{N}_0$ consists of $K$-valued multiple Wiener-Itô integrals of order $p$. The product formula~\eqref{multiplicationformulamultint} furthermore shows that each of these eigenfunctions is chaotic in the sense of Definition~\ref{def:2}. The carré du champ operator is given by $\Gamma(F,G)=\left\langle DF,DG \right\rangle_{\mathfrak{H}}$, where $\mathfrak{H}$ denotes the underlying Hilbert space on which the isonormal Gaussian process is defined. The operator $\Gamma(F,G)$ thus acts on $K$ via
\begin{equation*}
  \Gamma(F,G)u = \left\langle \left\langle DF,u \right\rangle_K,DG \right\rangle_{\mathfrak{H}}.
\end{equation*}

Using this concrete structure, our bounds can be expressed in terms of
kernel contractions. In applications, such contractions have proven to
be very useful, as they are typically easier to evaluate than moments
(see, among many others, \cite{nourdin-peccati:2009:steins-method-wiener,
  nourdin-peccati-podolskij:2011:quantitative-breuer-major-theorems}
for the context of Breuer-Major theorems, for instance).

Throughout the rest of this section, we assume a Dirichlet structure as introduced in the above paragraph as given.

We start by proving that the covariance condition~\eqref{eq:18} always holds in the present context.

\begin{lemma}
  \label{lem:2}
  For $p \in \N$ and $f \in \mathfrak{H}^{\odot p}\otimes K$, let $F=I_p(f)$ be a multiple integral with values in $K$. Then $F$ satisfies the covariance condition~\eqref{eq:18}.
\end{lemma}

\begin{proof}
  Let $u,v \in K$ . For better legibility, we will write $I_p(f_u) =I_p (\left\langle f,u \right\rangle)$ and $I_p(f_{v}) = I_p(\left\langle f,v \right\rangle)$. By the product formula for multiple integrals, we get that
\begin{align*}
  \E \left( I_p(f_u)^2I_p(f_{v})^2 \right)
  &= \notag
  \sum_{r=0}^{p} a_{p,r}^2 (2p-2r)!  \left\langle f_u \widetilde{\otimes}_r f_v, f_u \widetilde{\otimes}_r f_v \right\rangle_{\mathfrak{H}^{\otimes (2p-2r)}}
  \\ &=
       (2p)! \left\lVert f_u \widetilde{\otimes}_{} f_v \right\rVert_{\mathfrak{H}^{\otimes 2p}}^2
       +
       (p!)^2 \left\langle f_u,f_v \right\rangle_{\mathfrak{H}^{\otimes p}}^2
       \\ & \qquad \qquad+ \notag
            \sum_{r=1}^{p-1} a_r^2 (2p-2r)!
            \left\lVert f_u \widetilde{\otimes}_r f_v \right\rVert_{\mathfrak{H}^{\otimes (2p-2r)}}^2,
\end{align*}
where $a_{p,r} = r! \binom{p}{r}^2$. A straightforward modification of the calculations given in~\cite[Pages 97-98]{nourdin-peccati:2012:normal-approximations-malliavin} yields
\begin{align*}
  (2p)! &\left\lVert f_u \widetilde{\otimes}_{ } f_v \right\rVert_{\mathfrak{H}^{\otimes 2p}}^2
  \\ &
  =
     p!^2
       \left\lVert f_u \right\rVert_{\mathfrak{H}^{\otimes p}}^2
       \left\lVert f_v \right\rVert_{\mathfrak{H}^{\otimes p}}^2
       +
       p!^2
       \left\langle f_u,f_v \right\rangle_{\mathfrak{H}^{\otimes p}}^2
       +
       p!^{2}
       \sum_{r=1}^{p-1}
       \binom{p}{r}^2
       \left\lVert f_u \otimes_r f_v \right\rVert_{\mathfrak{H}^{\otimes (2p-2r)}}^2
       \\ &
            =
            \E \left( I_p(f_u)^2 \right)
            \E \left( I_p(f_v)^2 \right)
            +
            \E \left( I_p(f_u) I_p(f_v) \right)^2
       +
       p!^{2}
       \sum_{r=1}^{p-1}
       \binom{p}{r}^2
       \left\lVert f_u \otimes_r f_v \right\rVert_{\mathfrak{H}^{\otimes (2p-2r)}}^2.
\end{align*}
Combining the last two calculations and rearranging terms gives
\begin{multline}
  \label{symcontandnonsym}
  \E \left( I_p(f_u)^2 I_p(f_v)^2 \right)
  -
  \E \left( I_p(f_u)^2 \right)
  \E \left( I_p(f_v)^2 \right)
  -
  2
  \E \left( I_p(f_u)I_p(f_v) \right)^2
\\ =
  \sum_{r=1}^{p-1}
  \left(
  a_r^2 (2p-2r)!
      \left\lVert f_u \widetilde{\otimes}_r f_v \right\rVert_{\mathfrak{H}^{\otimes (2p-2r)}}^2
      +
      p!^{2}
   \binom{p}{r}^2
       \left\lVert f_u \otimes_r f_v \right\rVert_{\mathfrak{H}^{\otimes (2p-2r)}}^2
        \right) >0,
      \end{multline}
      which is the desired inequality.
\end{proof}
The contraction bound is as follows.
\begin{theorem}
 \label{thm:15}
 For $F \in \mathbb{D}^{1,4}$ with covariance operator $S$ and chaos decomposition $F = \sum_{p=1}^{\infty} F_p$, where $F_p=I_p(f)$ and $f \in \mathfrak{H}^{\odot p} \otimes K$, one has
\begin{equation}
  \left\lVert \left\langle DF,-DL^{-1}F \right\rangle_{\mathfrak{H}} - S \right\rVert_{L^2(\Omega;\on{HS}(K))}
  \leq
\widetilde{M}(F)
    +
    \widetilde{C}(F),
\end{equation}
where
\begin{align*}
  \widetilde{M}(F) &=
                     \sum_{p=1}^{\infty}
  \sqrt{
       \sum_{r=1}^{p-1}
c_{p,p}(r)^2
\left\lVert f_{p} \otimes_r f_{p} \right\rVert_{\mathfrak{H}^{\otimes (2p-2r)} \otimes K^{\otimes 2}}^2},
  \\
  \widetilde{C}(F) &=
                     \sum_{\substack{1 \leq p,q \leq \infty\\p \neq q}}
\sqrt{
  \sum_{r=1}^{p \land q}
  c_{p,q}(r)^2
  \left\lVert f_{p} \otimes_r f_{q} \right\rVert_{\mathfrak{H}^{\otimes (p+q-2r)} \otimes K^{\otimes 2}}^2}.
\end{align*}
and the positive constants $c_{p,q}(r)$ are given by
\begin{equation}
  \label{eq:21}
  c_{p,q}(r) = p^2(r-1)! \binom{p-1}{r-1} \binom{q-1}{r-1} (p+q-2r)!.
\end{equation}
\end{theorem}

\begin{proof}
  Let $\left\{ e_i \colon i \in \N \right\}$ be an orthonormal basis of $K$ and abbreviate the inner products $\left\langle F_p,e_i \right\rangle$ and $\left\langle f_p,e_i \right\rangle$ by $F_{p,i}$ and $f_{p,i}$, respectively.
  As in the proof of Theorem~\ref{thm:14}, it follows that
\begin{equation*}
  \left\lVert
    \left\langle DF,-DL^{-1}F \right\rangle_{\mathfrak{H}} - S
  \right\rVert_{L^2(\Omega;\on{HS}(K))}
  \leq
  \sum_{p,q=1}^{\infty}
  \sqrt{
    \sum_{i,j=1}^{\infty}
    \var \left( \left\langle D F_{p,i},-DL^{-1} F_{q,j} \right\rangle_{\mathfrak{H}} \right)
  },
\end{equation*}
so that we can apply the finite-dimensional result~\cite[Lemma 6.2.1]{nourdin-peccati:2012:normal-approximations-malliavin} and obtain
\begin{align}\notag
  \var \left( \left\langle D F_{p,i},-DL^{-1} F_{q,j} \right\rangle_{\mathfrak{H}} \right)
  \\ &= \label{eq:36}
  \begin{cases}
      \sum_{r=1}^{p \land q} c_{p,q}(r)^2
  \left\lVert f_{p,i} \widetilde{\otimes}_r f_{q,j} \right\rVert_{\mathfrak{H}^{\otimes (p+q-2r)}}^2
  & \qquad
  \text{if $p \neq q$,} \\
      \sum_{r=1}^{p -1} c_{p,p}(r)^2
      \left\lVert f_{p,i} \widetilde{\otimes}_r f_{p,j} \right\rVert_{\mathfrak{H}^{\otimes (2p-2r)}}^2
      & \qquad
      \text{if $p=q$,}
      \end{cases}
      \\ &\leq \notag
p^2
  \begin{cases}
      \sum_{r=1}^{p \land q} c_{p,q}(r)^2
  \left\lVert f_{p,i} \otimes_r f_{q,j} \right\rVert_{\mathfrak{H}^{\otimes (p+q-2r)}}^2
  & \qquad
  \text{if $p \neq q$,} \\
      \sum_{r=1}^{p -1} c_{p,p}(r)^2
      \left\lVert f_{p,i} \otimes_r f_{p,j} \right\rVert_{\mathfrak{H}^{\otimes (2p-2r)}}^2
      & \qquad
      \text{if $p=q$.}
  \end{cases}
\end{align}
The assertion follows after summing over $i$ and $j$.
\end{proof}
Combined with Theorem~\ref{thm:1}, the contraction bound just obtained yields the following result.
\begin{theorem}
  \label{thm:7}
  Let $Z$ be a centered Gaussian random variable on $K$ with covariance
  operator $S$ and $F \in \mathbb{D}^{1,4}$  with covariance operator $T$ and chaos decomposition $F= \sum_{p=1}^{\infty} I_p(f_p)$, where, for each $p \geq 1$, $f_p\in \mathfrak{H}^{\odot p}\otimes K$. Then
\begin{equation*}
  d_2(F,Z)
       \leq
         \frac{1}{2}
         \left( \widetilde{M}(F)
         +
         \widetilde{C}(F)
         +
         \left\lVert S-T \right\rVert_{\on{HS}}
         \right),
\end{equation*}
where the quantities $\widetilde{M}(F)$ and $\widetilde{C}(F)$ are defined as in Theorem~\ref{thm:15}.
\end{theorem}
As special cases for $K=\mathbb{\R}$, Theorem~\ref{thm:7} includes the
main results of~\cite{fissler-thaele:2016:new-quantitative-central}
and~\cite{nourdin-peccati-reveillac:2010:multivariate-normal-approximation}
(as usual in finite dimension, $d_2$ can be replaced by the
Wasserstein distance -- see the proof of Theorem \ref{thm:1}).

Let us now show how the results proved in Section~\ref{sec:fourth-moment-bounds} can be refined in the Wiener chaos setting. We start with the Fourth Moment Theorem.

\begin{theorem}[Infinite-dimensional Fourth Moment Theorem]
  \label{thm:10}
  Let $Z$ be a centered Gaussian random variable on $K$ with covariance
  operator $S$, and, for $p \geq 1$, let $\left\{ F_n \colon n \in \N
  \right\} = \left\{ I_p(f_n) \colon n \in \N \right\}$ be a sequence of $K$-valued multiple integrals
  such that $
  \on{tr}(S_n - S) \to 0$ as $n \to \infty$. Then, as $n \to \infty$, the following assertions are equivalent.
  \begin{enumerate}[(i)]
  \item\label{item:1}
    $F_n$ converges in distribution to $Z$,
\item\label{item:2} $\E \left( \left\lVert F_n\right\rVert^4 \right)  \to \E \left( \left\lVert Z \right\rVert^4 \right)$,
\item\label{item:3}
   $\left\lVert f_n \otimes_r f_n \right\rVert_{\mathfrak{H}^{(2p-2r)} \otimes K^{\otimes 2}} \to 0$ for all $r=1,\dots,p-1$,
\item\label{item:4} $\left\lVert f_n \widetilde{\otimes}_r f_n \right\rVert_{\mathfrak{H}^{(2p-2r)} \otimes K^{\otimes 2}} \to 0$ for all $r=1,\dots,p-1$,
\item\label{item:5}
  $\left\lVert \left\langle DF_n,D F_n \right\rangle_{\mathfrak{H}} - pS_n \right\rVert_{L^2(\Omega;\on{HS}(K))} \to 0$.
  \end{enumerate}
\end{theorem}
\begin{proof}
As $ \on{tr}(S_{n} - S) \to 0$ as $n \to \infty$, hypercontractivity of Wiener chaos implies that for any $ r\geq 2$, $\sup_n
\Ex{\norm{F_n}^r} < \infty$, which yields that~\emph{(\ref{item:1})}
implies~\emph{(\ref{item:2})} by uniform integrability. Summing~\eqref{symcontandnonsym} over $i$ and $j$ and using~\eqref{eq:22} yields the implication~\emph{(\ref{item:2})}$\Rightarrow$\emph{(\ref{item:3})} (and also~\emph{(\ref{item:2})}$\Rightarrow$\emph{(\ref{item:4})}). The fact that
 $\left\lVert f_n
  \widetilde{\otimes}_r f_n\right\rVert_{\mathfrak{H}^{(2p-2r)}
  \otimes K^{\otimes 2}} \leq \left\lVert f_n \otimes_r f_n \right\rVert_{\mathfrak{H}^{(2p-2r)} \otimes K^{\otimes 2}}$ gives \emph{(\ref{item:3})}$\Rightarrow$\emph{(\ref{item:4})} and the implication
\emph{(\ref{item:4})}$\Rightarrow$\emph{(\ref{item:5})}
follows by summing~\eqref{eq:36} over $i$ and $j$. Finally, \emph{(\ref{item:5})}$\Rightarrow$\emph{(\ref{item:1})} is a consequence of Theorem~\ref{thm:15}.
\end{proof}
The corresponding Fourth Moment Theorems for random variables with infinite chaos expansion (Theorem~\ref{thm:12} in Section~\ref{sec:fourth-moment-bounds}) can be expressed using contractions as follows:

\begin{theorem}
Let $\left\{ F_n \colon n \in \N \right\}$ be a sequence of square integrable $K$-valued random variables with chaos decomposition
\begin{equation}
  \label{infinitechaosexpforreffinitedimexamples}
F_n = \sum_{p=1}^{\infty}I_p(f_{p,n}),
\end{equation}
where, for each $n,p \geq 1$, $f_{p,n} \in \mathfrak{H}^{\odot
  p}\otimes K$. Suppose that:
\begin{enumerate}[(i)]
\item for every $p  \in \mathbb{N}$ there exists $f_p \in \mathfrak{H}^{\odot p} \otimes K$ such that $\left\lVert f_{n,p} - f_p \right\rVert_{\mathfrak{H}^{\otimes p} \otimes K} \to 0$,
  \begin{equation*}
    \sum_{p=1}^{\infty} p! \left\lVert f_p  \right\rVert_{\mathfrak{H}^{\otimes p} \otimes K}^2 < \infty
  \end{equation*}
  and
\begin{equation*}
  \lim_{N \to \infty} \sup_{n \geq 1} \sum_{p=N+1}^{\infty} p! \left\lVert f_{p,n} \right\rVert_{\mathfrak{H}^{\otimes p}\otimes K}^2 = 0.
\end{equation*}
\item for all $p \in \N$ and $r=1,\dots,p-1$, it holds that
\begin{equation*}
  \left\lVert f_{p,n} \otimes_{r} f_{p,n} \right\rVert_{\mathfrak{H}^{\otimes 2(p-r)} \otimes K^{\otimes 2}} \to 0.
\end{equation*}
\end{enumerate}
Then $F_n$ converges in distribution to a centered Gaussian $Z$ with covariance operator $S$ given by
\begin{equation*}
  S = \sum_{p=1}^{\infty} \E \left( \left\lVert f_p \right\rVert^2_{\mathfrak{H}^{\otimes p}} \right),
\end{equation*}
where, with some slight abuse of notation, $\E \left( \left\lVert f_p \right\rVert_{\mathfrak{H}^{\otimes p}}^2 \right) \in K \otimes K \simeq L(K,K)$ denotes the covariance operator of $I_p(f_p)$.
\end{theorem}
\begin{proof}
  For $p,n \in \mathbb{N}$, let $S_p$ and $S_{p,n}$ be the covariance operators of $I_p(f_p)$ and $I_p(f_{p,n})$, respectively. Then
\begin{align*}
  \left| \on{tr}(S_{p,n}-S_p) \right| &=
                                        \left|
                         \E \left(
                         \left\lVert I_p(f_{p,n}) \right\rVert^2_K
                         -
                         \left\lVert I_p(f_p) \right\rVert^2_K
                                        \right)
                                         \right|
  \\ &\leq
       \sqrt{
       \E \left(
       \left\lVert I_p(f_{p,n}- I_p(f_p) \right\rVert^2_K
       \right)
       \E \left(
       \left\lVert I_p(f_{p,n}) + I_p(f_p) \right\rVert^2_K
       \right)
       }
  \\ &=
       p!
       \left\lVert f_{p,n}-f_p \right\rVert_{\mathfrak{H}^{\otimes p} \otimes K}
       \left\lVert f_{p,n}+f_p \right\rVert_{\mathfrak{H}^{\otimes p} \otimes K},
\end{align*}
which tends to zero as $n \to \infty$ by assumption $(i)$. As $\on{tr}(S_{p,n}) = \E \left( \left\lVert I_{p,n} \right\rVert_K^2 \right) = p! \left\lVert f_{p,n} \right\rVert_{\mathfrak{H}^{\otimes p}\otimes K}^2$, the same assumption also implies that
\begin{equation*}
  \lim_{N \to \infty} \sup_{n \geq 1} \sum_{p=N+1}^{\infty} \on{tr}(S_{p,n}) = 0.
\end{equation*}
The rest of the proof can now be done as in Theorem~\ref{thm:12}, using the  bound provided by Theorem~\ref{thm:15}.
\end{proof}

\section{Quantifying the functional Breuer-Major Theorem}
\label{sec:quant-funct-breu}

In this section, we will give rates of convergence for a functional
version of the seminal Breuer-Major Theorem. To introduce the setting, let $X= \left\{ X_t \colon t \geq 0 \right\}$ be a centered, stationary Gaussian process and define $\rho(k)= \E \left( X_0X_k \right)$ such that $\E \left( X_s X_t \right) = \rho(t-s)=\rho(s-t)$. Assume $\rho(0)=1$, denote the standard Gaussian measure on $\R$ by $\gamma$ and let $\varphi \in L^2(\R,\gamma)$ be of Hermite rank $d \geq 1$, so that $\varphi$ can be expanded in the form
\begin{equation}
 \label{eq:44}
  \varphi(x) = \sum_{i=d}^{\infty} c_i H_i(x), \qquad c_d \neq 0,
\end{equation}
where $H_i(x) =  (-1)^{i} \me^{x^2/2} \left( \frac{d}{dx} \right)^i
\me^{-x^2/2}$ is the $i$th Hermite polynomial. The Breuer-Major Theorem then states that under the condition
\begin{equation*}
  \sum_{k \in \Z}^{} \rho(k)^d < \infty,
\end{equation*}
the finite-dimensional distributions of the stochastic process
$\left\{ U_n(t) \colon t \in [0,1] \right\}$ given by
\begin{equation}
  \label{eq:47}
  U_n(t) = \frac{1}{\sqrt{n}}
  \sum_{i=0}^{\lfloor nt \rfloor} \varphi(X_i)
\end{equation}
converge in law to those of a scaled Brownian motion $\sigma W$, where
$W=\left\{ W_t \colon t \in [0,1] \right\}$  is a standard Brownian motion and the scaling is given by
\begin{equation}
  \label{eq:67}
  \sigma^2 = \sum_{p=d}^{\infty} p! c_p^2 \sum_{k \in \Z}^{} \rho(k)^p.
\end{equation}
After its discovery by Breuer and Major (see~\cite{breuer-major:1983:central-limit-theorems}), it took more than twenty years until progress was made towards quantifying this result. Taking $X$ to be the normalized increment process of a fractional Brownian motion, Nourdin and Peccati (\cite{nourdin-peccati:2009:steins-method-wiener}), as an illustration of the Malliavin-Stein method introduced in the same reference, were able to associate rates to the normal convergence of the \emph{chaotic projections} of the coordinate sequences of $U_n$, i.e., to the random sequence
\begin{equation}
  \label{eq:59}
  \frac{1}{\sqrt{n}} \sum_{k=1}^{n} H_p \left( n^H
    \left(
       B_{\frac{k+1}{n}}^H - B_{\frac{k}{n}}^H
    \right)
  \right),
\end{equation}
where $H_p$ denotes the $p$th Hermite polynomial and $B^H$ is a fractional Brownian motion with Hurst index $H$. Note that the random variables defined in~\eqref{eq:59} can be represented as multiple integrals of order $p$ and therefore are elements of the $p$th Wiener chaos. Recently, the Breuer-Major Theorem has been intensively studied, and very strong results have been obtained concerning the coordinate sequence, providing rates of convergence in total variation distance for general functions $\varphi$ under rather weak assumptions (see~\cite{%
  nualart-zhou:2018:total-variation-estimates%
  ,nourdin-nualart-peccati:2019:breuer-major-theorem-total%
  ,nourdin-peccati-yang:2019:berry-esseen-bounds-breuer-major}).
Turning to infinite-dimension, it also has been proved recently
in~\cite{campese-nourdin-nualart:2018:continuous-breuer-major-theorem} and~\cite{nourdin-nualart:2018:functional-breuer-major-theorem} that the process $U_n$ converges in distribution towards a scaled Brownian motion in the Skorohod space or in the space of continuous functions (replacing the Gauss brackets in the sum by a linear interpolation).

In this section, it will be shown how, using our bounds, one can associate rates to the aforementioned functional convergences, taking place in a suitable Hilbert space $K$ containing $D([0,1])$ and $C_0([0,1])$, respectively. The rates are obtained through the contraction bounds obtained in the previous section, which allow a natural and straightforward lifting of the one-dimensional results.
We illustrate this method on~\cite[Example
2.5]{nourdin-peccati-podolskij:2011:quantitative-breuer-major-theorems},
where $\varphi=H_p$ and $\rho(k) = k^{\alpha} l(k)$ for some $\alpha<
0$ and a slowly varying function $l$. This latter assumption on $\rho$
for example includes the case where $X$ is the increment process of a
fractional Brownian motion. Also, for simplicity, we set
$K=L^2([0,1])$. Our results also allow the analysis of more general
functions $\varphi$ and smaller Hilbert spaces $K$ with finer
topologies, such as the Besov-Liouville
(see~~\cite{samko-kilbas-marichev:1993:fractional-integrals-derivatives}
for definitions
and~\cite{coutin-decreusefond:2013:steins-method-brownian} for proofs
of related functional limit theorems in this space) or other
reproducing kernel Hilbert spaces, but as the calculations are more
involved and also quite lenghty and technical, we decided to focus on
the general picture in this article and will provide full details on
this topic in a dedicated followup work.

~\newline
The statement is as follows.

\begin{theorem}
  \label{thm:13}
  Let $\left\{U_n(t) \colon t \in [0,1] \right\}$ be the stochastic process defined in~\eqref{eq:47}, considered as a sequence of random variables taking values in $L^2([0,1])$,   assume that $\varphi=H_p$ for some $p  \in \mathbb{N}$ and that the covariance function $\rho$ of the underlying centered, stationary Gaussian process is of the form $\rho(k)=\left| k \right|^{\alpha} l(\left| k \right|)$, where $\alpha<-1/p$ and $l$ is a slowly varying function. Then there exists a constant $C > 0$, such that

\begin{equation}
 \label{eq:60}
  d_2(U_n,\sigma W) \leq C r_{\alpha}(n)
\end{equation}
where $\sigma$ is defined in~\eqref{eq:67}, $W$ denotes a standard Brownian motion on $L^2([0,1])$ and the rate function is given by
\begin{equation*}
  r_{\alpha}(n) =
  \begin{cases}
    n^{-1/2}  &\qquad \qquad \text{if $\alpha < -1$}, \\
    n^{\alpha/2}l(n) & \qquad \qquad \text{if $\alpha \in \left( -1, -\frac{1}{p-1} \right)$},\\
    n^{(\alpha q+1)/2 } l^2(n) & \qquad \qquad \text{if $\alpha \in \left( -\frac{1}{p-1}, \frac{-1}{p} \right)$.}
  \end{cases}
\end{equation*}
\end{theorem}

\begin{remark}
  \label{rmk:3}
  Theorem~\ref{thm:13} applies to the case where $X_i=B_{i+1}^H - B_i^H$ is the increment process of a fractional Brownian motion with Hurst index $H \leq \frac{2p-1}{2p}$. In this case, the corresponding bound reads
\begin{equation*}
  d_2(U_n,\sigma W) \leq C
  \begin{cases}
    n^{-1/2} & \qquad \qquad \text{if $H \in \left(0,\frac{1}{2}\right)$,}\\
    n^{H-1} & \qquad \qquad \text{if $H  \in  \left[ \frac{1}{2}, \frac{2p-3}{2p-2} \right]$}\\
    n^{(2pH-2p+1)/2} & \qquad \qquad \text{if $H  \in  \left( \frac{2p-3}{2p-2}, \frac{2p-1}{2p} \right)$}.
  \end{cases}
\end{equation*}
See~\cite[Example
2.6]{nourdin-peccati-podolskij:2011:quantitative-breuer-major-theorems}
for further details. See also \cite[Exercise
7.5.1]{nourdin-peccati:2012:normal-approximations-malliavin} for a
particular (and simpler) case where the function $\varphi$ is taken to be the $p$-th
Hermite polynomial $H_p$.
\end{remark}

\begin{proof}[Proof of Theorem~\ref{thm:13}]
  Throughout this proof, $C$ denotes a positive constant which might change from line to line.
  Let $\mathfrak{H}$ be the Hilbert space obtained by the closure of the set of all finite linear combinations of indicator functions $1_{[0,t]}$, $t \geq 0$ with respect to the inner product
\begin{equation*}
  \left\langle 1_{[0,s]}, 1_{[0,t]} \right\rangle_{\mathfrak{H}} = \rho ( t-s).
\end{equation*}
and let $\mathcal{X}$ be an isonormal Gaussian process on
$\mathfrak{H}$ (for details on this construction, see~\cite[Example 2.1.5]{nourdin-peccati:2012:normal-approximations-malliavin}).
Then
\begin{equation*}
  \E \left( \mathcal{X}(1_{[0,i]}) \mathcal{X}(1_{[0,j]}) \right)
  =
  \left\langle 1_{[0,i]},1_{[0,j]} \right\rangle_{\mathfrak{H}}
  =
  \rho ( j-i)
  =
  \E \left( X_iX_j \right),
\end{equation*}
where expectations are taken over the respective probability spaces of $\mathcal{X}$ and $X$. Furthermore, note that $U_n$ has the same law as
$I_p (f_{n,t})$, where
\begin{equation}
 \label{eq:54}
 f_{n,t}(x)
 = \frac{1}{\sqrt{n}}
 \sum_{i=1}^{\lfloor nt \rfloor -1}
 g_{p}(i,x)
 = \frac{1}{\sqrt{n}}
 \sum_{i=1}^{n-1}
 1_{[\frac{i+1}{n},1]}(t)
 g(i,x)
\end{equation}
and $g(i,x)=  \prod_{j=1}^p 1_{[0,i]}(x_j) \geq 0$.
Let us denote by $T_{n}$ the covariance operator of $U_{n}$ and define
\begin{equation*}
\sigma_{n}^2 = p! \sum_{k \in \mathbb{Z}} \rho(k)^p \left( 1 - \frac{\left| k \right|}{n} \right)
 1_{\left\{ \left| k \right| < n \right\}}.
\end{equation*}

Now
\begin{equation}
  \label{eq:25}
d_{2}
  \left( U_n,\sigma W \right)
  \leq
  d_{2}
  \left( U_n, \sigma_n W \right)
  +
  d_{2}
    \left( \sigma_nW,\sigma W \right),
  \end{equation}
and applying Corollary~\ref{cor:1} together with the identity~\eqref{eq:16}, we obtain
\begin{multline}
 \label{eq:26}
d_2(\sigma_nW,\sigma W)
 \le
 \frac{1}{2}
 \left\lVert \sigma_n^2 S - \sigma^2 S \right\rVert_{L^1(\Omega;\mathcal{S}_1(L^2([0,1])))}
 \\ =
 \frac{1}{2}
 \left| \sigma_n^2-\sigma^2 \right|
 \on{tr} S
 =
 \frac{1}{2}
 \left| \sigma_n^2 - \sigma^2 \right|
 \leq
 C
   \left( n^{-1} + n^{\alpha p+1}l(n) \right),
\end{multline}
 where the last inequality follows after a straightforward calculation, using the same estimate as in~\cite[Example 2.5]{nourdin-peccati-podolskij:2011:quantitative-breuer-major-theorems}.
Furthermore, by Theorem~\ref{thm:7},
\begin{equation}
  \label{eq:27}
  d_{2}
  \left( U_n,\sigma_n W \right)
 \leq
    \frac{1}{2}
\sum_{r=1}^{p -1} p
      \left\lVert f_{n,\cdot} \otimes_r f_{n,\cdot} \right\rVert_{\mathfrak{H}^{\otimes (2p-2r)} \otimes L^2([0,1])^{\otimes 2}}
+ \frac{1}{2}
       \left\lVert T_n - \sigma_n^2 S \right\rVert_{\on{HS}(L^2([0,1]))}.
\end{equation}
Lemma~\ref{lem:6} yields that
\begin{equation}
 \label{eq:37}
 \left\lVert T_n - \sigma_n^2 S \right\rVert_{HS}
 \leq
 C \,
   \left( n^{-1} + n^{\alpha p+1}l(n) \right).
  \end{equation}
Plugging~\eqref{eq:37} into~\eqref{eq:27}, then together
with~\eqref{eq:26} into~\eqref{eq:25} and noting that 
\begin{equation*}
\frac{n^{-1 \lor
  (1-\alpha p)}}{r_{\alpha}(n)} \to 0
\end{equation*}
as $n \to \infty$, it remains to show that
\begin{equation}
  \label{eq:30}
\sum_{r=1}^{p -1}
      \left\lVert f_{n,\cdot} \otimes_r f_{n,\cdot} \right\rVert_{\mathfrak{H}^{\otimes (2p-2r)} \otimes L^2([0,1])^{\otimes 2}}
\leq C \, r_{\alpha}(n).
\end{equation}
Now, as for any $s_1,s_2 \in [0,1]$, it holds that
\begin{equation*}
  \left\langle 1_{[\frac{s_1}{n},1]}(\cdot),1_{[\frac{s_2}{n},1]}(\cdot) \right\rangle_{L^2([0,1])}
  \leq 1,
\end{equation*}
we have that for $r=1,\ldots,p \land q$,
\begin{equation*}
  \left\lVert f_{n,t} \otimes_r f_{n,t} \right\rVert_{\mathfrak{H}^{\otimes (p+q-2r)} \otimes L^2([0,1])^{\otimes 2}}
  \leq
  \left\lVert f_{n,1} \otimes_r f_{n,1} \right\rVert_{\mathfrak{H}^{\otimes (p+q-2r)}}.
\end{equation*}
In other words, the contraction norms of the kernels of the stochastic
process $\left\{ U_n(t) \colon t \in [0,1] \right\}$ are bounded by those of the random variable $U_n(1)$, so that~\eqref{eq:30} follows from the one-dimensional calculations in~\cite[Example 2.5]{nourdin-peccati-podolskij:2011:quantitative-breuer-major-theorems}.
\end{proof}

\begin{lemma}
  \label{lem:6}
  In the setting of Theorem \ref{thm:13}, it holds that
  \begin{equation}
 \label{eq:57}
 \left\lVert T_n - \sigma_n^2 S \right\rVert_{\on{HS}(L^2([0,1]))} \leq
   C \,
  \left( n^{-1} + n^{\alpha p+1}l(n) \right).
\end{equation}
\end{lemma}

\begin{proof}
  The operator $K_n = T_{n}-\sigma_n^2 S$ is a Hilbert-Schmidt integral operator of the form $K_nf (t) = \int_0^1 k_n(s,t) f(s) \d{s}$, with kernel $k_n$ given by
  \begin{equation*}
    k_n(s,t) =
    \E \left( U_n(s) U_n(t) \right) - (s \land t) \sigma_n^2.
  \end{equation*}
   Note that by orthogonality
   \begin{equation}
      \label{eq:38}
\E \left( U_n(s) U_n(t) \right)
 =
      p!
       \left\langle
       f_{n,s}
       ,
       f_{n,t}
       \right\rangle_{\mathfrak{H}^{\otimes p}},
\end{equation}
where the kernels $f_{n,\cdot}$ are given by \eqref{eq:54}.
Now
\begin{align}
       \left\langle
       f_{p,n,s}
       ,
       f_{p,n,t}
  \right\rangle_{\mathfrak{H}^{\otimes p}}
  &=
    \notag
    \frac{1}{n}
 \sum_{i,j=1}^{n}
    1_{[\frac{i}{n},1]}(s)
    1_{[\frac{j}{n},1]}(t)
    \rho \left( \left| i-j \right| \right)^p
  \\ &=
       \notag
       \frac{1}{n}
       \sum_{i=1}^{n}
       \sum_{j=1-i}^{n-i}
      1_{[\frac{i}{n},1]}(s)
      1_{[\frac{j+i}{n},1]}(t)
       \rho \left( \left| j \right| \right)^p
  \\ &=
       \notag
       \frac{1}{n}
       \sum_{j=-(n-1)}^{n-1}
       \sum_{i=1}^{n}
       1_{[1-j,n-j]}(i)
      1_{[\frac{i}{n},1]}(s)
      1_{[\frac{j+i}{n},1]}(t)
       \rho (\left| j \right|)^p
  \\ &=
       \notag
       \frac{1}{n}
       \sum_{j=-(n-1)}^{n-1}
       \sum_{i=1}^{n}
       1_{[1-j,n-j]}(i)
      1_{[\frac{i}{n},1]}(s)
      1_{[\frac{j+i}{n},1]}(t)
       \rho (\left| j \right|)^p
  \\ &=
       \label{eq:43}
A_n + B_n + C_n,
\end{align}
where the terms $A_n$, $B_n$ and $C_n$ are obtained by decomposing the sum  over $j$ according to
\begin{equation*}
  \sum_{j=-(n-1)}^{n-1}\beta(j) = \sum_{j=-(n-1)}^{-1}\beta(j)  + \beta(0) + \sum_{j=1}^{n-1}\beta(j),
\end{equation*}
where 
\begin{equation*}
\beta(j) = \frac{1}{n}\sum_{i=1}^{n}
       1_{[1-j,n-j]}(i)
      1_{[\frac{i}{n},1]}(s)
      1_{[\frac{j+i}{n},1]}(t)
       \rho (\left| j \right|)^p.
\end{equation*}
Now, we have
\begin{align*}
  A_n &=
  \frac{1}{n}
  \sum_{j=-(n-1)}^{-1} \# \left\{  1 \leq i \leq n \colon 1-j\leq i, \, i \leq ns, \, i \leq nt-j \right\}
  \rho \left( \left| j \right| \right)^p
  \\ &=
  \frac{1}{n}
       \sum_{j=-(n-1)}^{-1}
     \left(  \lfloor ns \land  \left( nt - j \right)\rfloor +j \right)
       \rho \left( \left| j \right| \right)^p
  \\ &=
       \sum_{j=-(n-1)}^{-1}
       \rho \left( \left| j \right| \right)^p
       \times
       \begin{cases}
         \frac{\lfloor ns \rfloor}{n} + \frac{j}{n} &  \text{if $t-s > \frac{j}{n}$}
         \\
         \frac{\lfloor nt \rfloor}{n} &   \text{if $t-s \leq \frac{j}{n}$}
       \end{cases},
  \\
  B_n &= \frac{1}{n}  \# \left\{  1 \leq i \leq n \colon i \leq n (s
      \land t) \right\}
      =
      \frac{\lfloor n (s \land t) \rfloor}{n}
  \\
  \intertext{and}
  C_n &= \frac{1}{n}
  \sum_{j=1}^{n-1} \# \left\{  1 \leq i \leq n \colon i \leq n-j, \, i \leq ns, \, i \leq nt-j \right\}
      \rho \left( \left| j \right| \right)^p
  \\ &= \frac{1}{n}
  \sum_{j=1}^{n-1} \left( (n-j) \land \lfloor ns \land (nt-j) \rfloor \right)
       \rho \left( \left| j \right| \right)^p
       \\ &= \sum_{j=1}^{n-1}
            \rho \left( \left| j \right| \right)^p \times
            \begin{cases}
         \frac{\lfloor ns \rfloor}{n}  &  \text{if $t-s >
           \frac{j}{n}$}
         \\
         \frac{\lfloor nt \rfloor}{n} - \frac{j}{n}  &   \text{if $t-s \leq
           \frac{j}{n}$}
       \end{cases}.
\end{align*}
Plugging~\eqref{eq:43} into~\eqref{eq:38} and using formula~\eqref{eq:32} for $\sigma_n$,
this yields
\begin{multline*}
  \E \left( U_n(s)U_n(t) \right)
  -
  (s \land t)\sigma_n
  \\ =
  p!
  \left(
    A_n + B_n + C_n -   (s \land t)
      \sum_{j=-(n-1)}^{n-1}
  \rho \left( \left| j \right| \right)^p
  \left( 1 - \frac{\left| j \right|}{n} \right)
  \right).
\end{multline*}
and after a tedious but straightforward calculation (similarly as in
\cite[Proof of Theorem 2.2]{nourdin-peccati-podolskij:2011:quantitative-breuer-major-theorems} ), one arrives at
\begin{align*}
  \left| k_n(s,t) \right|
  =
   \left|
  \E \left( U_n(s)U_n(t) \right)
  -
  (s \land t)\sigma_n
  \right|
  &\lesssim \frac{1}{n}
  \left(
  1
  +
  p! \sum_{j=1}^{n-1} j \rho(\left| j \right|) ^p
  \right)
  \\ &\lesssim
       n^{-1} + n^{\alpha p + 1}l(n),
\end{align*}
where we have used Karamata's theorem to obtain the last estimate
(see~\cite[Example
2.5]{nourdin-peccati-podolskij:2011:quantitative-breuer-major-theorems}
for details).
Consequently,
\begin{equation*}
  \left\lVert T_n - \sigma_n^2 S \right\rVert_{\on{HS}(L^2([0,1]))}
  =
  \left\lVert k_n \right\rVert_{L^2([0,1]^2)}
  \leq
  \sup_{s,t \in [0,1]} \left| k_n(s,t) \right|
  \leq
  C \,
  \left( n^{-1} + n^{\alpha p+1}l(n) \right)
  \end{equation*}
  as asserted.
\end{proof}


\begin{thebibliography}{10}

\bibitem{araujo-gine:1980:central-limit-theorem}
A. Araujo and E. Gin\'{e}, \emph{The central limit theorem for real
  and {B}anach valued random variables}, John Wiley \& Sons, New
  York-Chichester-Brisbane, 1980, Wiley Series in Probability and Mathematical
  Statistics. \MR{576407}

\bibitem{azmoodeh-campese-poly:2014:fourth-moment-theorems}
E. Azmoodeh, S. Campese, and G. Poly, \emph{Fourth {M}oment
  {T}heorems for {M}arkov diffusion generators}, J. Funct. Anal. \textbf{266}
  (2014), no.~4, 2341--2359. \MR{3150163}

\bibitem{bakry-gentil-ledoux:2014:analysis-geometry-markov}
D. Bakry, I. Gentil, and M. Ledoux, \emph{Analysis and geometry of
  {Markov} diffusion operators}, Grundlehren der {Mathematischen}
  {Wissenschaften} [{Fundamental} {Principles} of {Mathematical} {Sciences}],
  vol. 348, Springer, Cham, 2014. \MR{3155209}.

\bibitem{barbour:1990:steins-method-diffusion}
A.~D. Barbour, \emph{Stein's method for diffusion approximations}, Probab. Theory Related Fields \textbf{84} (1990), no.~3, 297--322. \MR{1035659}

\bibitem{bogachev:1998:gaussian-measures}
V.~I. Bogachev, \emph{Gaussian measures}, Mathematical {Surveys} and
  {Monographs}, vol.~62, American Mathematical Society, Providence, RI, 1998.
  \MR{1642391}

\bibitem{bouleau-hirsch:1991:dirichlet-forms-analysis}
N. Bouleau and F. Hirsch, \emph{Dirichlet forms and analysis on
  {Wiener} space}, de {Gruyter} {Studies} in {Mathematics}, vol.~14, Walter de
  Gruyter \& Co., Berlin, 1991. \MR{1133391}

\bibitem{bourguin-campese-leonenko-ea:2019:four-moments-theorems}
S. Bourguin, S. Campese, N. Leonenko, and M.~S. Taqqu,
  \emph{Four moments theorems on {M}arkov chaos}, Ann. Probab. \textbf{47}
  (2019), no.~3, 1417--1446. \MR{3945750}

\bibitem{breuer-major:1983:central-limit-theorems}
P. Breuer and P. Major, \emph{Central limit theorems for nonlinear
  functionals of {Gaussian} fields}, J. Multivariate Anal.
  \textbf{13} (1983), no.~3, 425--441. \MR{716933}

\bibitem{campese-nourdin-nualart:2018:continuous-breuer-major-theorem}
S.~{Campese}, I.~{Nourdin}, and D.~{Nualart}, \emph{Continuous {B}reuer-{M}ajor theorem: tightness and nonstationarity}, Ann. Probab. \textbf{48}
  (2020), no.~1, 147--177. \MR{4079433}

\bibitem{campese-nourdin-peccati-ea:2016:multivariate-gaussian-approximations}
S. Campese, I. Nourdin, G. Peccati, and G. Poly,
  \emph{Multivariate {G}aussian approximations on {M}arkov chaoses}, Electron.
  Commun. Probab. \textbf{21} (2016), Paper No. 48, 9. \MR{3522594}

\bibitem{carmona-tehranchi:2006:interest-rate-models}
R.~A. Carmona and M.~R. Tehranchi, \emph{Interest rate models: an
  infinite dimensional stochastic analysis perspective}, Springer Finance,
  Springer-Verlag, Berlin, 2006. \MR{2235463}

\bibitem{coutin-decreusefond:2013:steins-method-brownian}
L.~Coutin and L.~Decreusefond, \emph{Stein's method for {B}rownian
  approximations}, Commun. Stoch. Anal. \textbf{7} (2013), no.~3, 349--372.
  \MR{3167403}

\bibitem{fissler-thaele:2016:new-quantitative-central}
T.~{Fissler} and C.~{Thaele}, \emph{A new quantitative central limit theorem on
  the wiener space with applications to Gaussian processes}, \ARXIV{1610.01456}

\bibitem{fukushima-oshima-takeda:2011:dirichlet-forms-symmetric}
M. Fukushima, Y. Oshima, and M. Takeda, \emph{Dirichlet forms
  and symmetric {Markov} processes}, extended ed., de {Gruyter} {Studies} in
  {Mathematics}, vol.~19, Walter de Gruyter \& Co., Berlin, 2011. \MR{2778606}

\bibitem{grorud-pardoux:1992:integrales-hilbertiennes-anticipantes}
A.~Grorud and E. Pardoux, \emph{Int\'{e}grales Hilbertiennes anticipantes
  par rapport \`a un processus de {W}iener cylindrique et calcul stochastique
  associ\'{e}}, Appl. Math. Optim. \textbf{25} (1992), no.~1, 31--49.
  \MR{1133251}

\bibitem{gross:1967:potential-theory-hilbert}
L. Gross, \emph{Potential theory on {H}ilbert space}, J. Functional
  Analysis \textbf{1} (1967), 123--181. \MR{0227747}

\bibitem{kannan-kannappan:1975:characterization-gaussian-measures}
D.~Kannan and P. Kannappan, \emph{On a characterization of {G}aussian measures
  in a {H}ilbert {S}pace}, Ann. Inst. H. Poincar\'{e} Sect. B. (N.S.)
  \textbf{11} (1975), no.~4, 397--404 (1976). \MR{0394787}

\bibitem{kasprzak:2017:steins-method-multivariate}
M.~{Kasprzak}, \emph{Stein's method for multivariate brownian approximations of
  sums under dependence}, \ARXIV{1708.02521}

\bibitem{kasprzak:2017:diffusion-approximations-via}
M. Kasprzak, \emph{Diffusion approximations via stein's method and time
  changes}, \ARXIV{1701.07633}

\bibitem{kasprzak:2017:multivariate-functional-approximations}
M. Kasprzak, \emph{Multivariate functional approximations with stein's method of
  exchangeable pairs}, \ARXIV{1710.09263}

\bibitem{kruse:2014:strong-weak-approximation}
R. Kruse, \emph{Strong and weak approximation of semilinear stochastic
  evolution equations}, Lecture Notes in Mathematics, vol. 2093, Springer,
  Cham, 2014. \MR{3154916}

\bibitem{kuelbs:1970:gaussian-measures-banach}
J.~Kuelbs, \emph{Gaussian measures on a {B}anach space}, J. Functional
Analysis
  \textbf{5} (1970), 354--367. \MR{0260010}

\bibitem{kuo:1975:gaussian-measures-banach}
H.H. Kuo, \emph{Gaussian measures in {B}anach spaces}, Lecture Notes in
  Mathematics, Vol. 463, Springer-Verlag, Berlin-New York, 1975. \MR{0461643}

\bibitem{ledoux:2012:chaos-markov-operator}
M.~Ledoux, \emph{Chaos of a markov operator and the fourth moment condition},
  Ann. Probab. \textbf{40} (2012), no.~6, 2439--2459. \MR{3050508}

\bibitem{leon-nualart:1998:stochastic-evolution-equations}
J.~A. Le\'{o}n and D. Nualart, \emph{Stochastic evolution equations with
  random generators}, Ann. Probab. \textbf{26} (1998), no.~1, 149--186.
  \MR{1617045}

\bibitem{ma-rockner:1992:introduction-theory-nonsymmetric}
Z.~M. Ma and M. R\"{o}ckner, \emph{Introduction to the theory of
  (nonsymmetric) {D}irichlet forms}, Universitext, Springer-Verlag, Berlin,
  1992. \MR{1214375}

\bibitem{noreddine-nourdin:2011:gaussian-approximation-vector-valued}
S. Noreddine and I. Nourdin, \emph{On the {Gaussian} approximation of
  vector-valued multiple integrals}, J. Multivariate Anal.
  \textbf{102} (2011), no.~6, 1008--1017. \MR{2793872}

\bibitem{nourdin-nualart:2018:functional-breuer-major-theorem}
I.~{Nourdin} and D.~{Nualart}, \emph{The functional breuer-major theorem},
  Probab. Theory Related Fields \textbf{176} (2020), no.~1-2, 203--218. \MR{4055189}

  \bibitem{nourdin-nualart-peccati:2019:breuer-major-theorem-total}
I.~{Nourdin}, D.~{Nualart}, and G.~{Peccati}, \emph{The breuer-major theorem in
  total variation: improved rates under minimal regularity},
Stochastic Process. Appl. \textbf{131} (2021), 1--20. \MR{4151212}

\bibitem{nourdin-peccati:2009:steins-method-wiener}
I. Nourdin and G. Peccati, \emph{Stein's method on {Wiener} chaos},
  Probab. Theory Related Fields \textbf{145} (2009), no.~1-2,
  75--118. \MR{2520122}

\bibitem{nourdin-peccati:2012:normal-approximations-malliavin}
I. Nourdin and G. Peccati, \emph{Normal {Approximations} with {Malliavin} {Calculus}: {From}
  {Stein}'s {Method} to {Universality}}, 1 ed., Cambridge University Press, May
  2012. \MR{2962301}

\bibitem{nourdin-peccati-podolskij:2011:quantitative-breuer-major-theorems}
I. Nourdin, G. Peccati, and M. Podolskij, \emph{Quantitative
  {Breuer}-{Major} theorems}, Stochastic Process. Appl.
  \textbf{121} (2011), no.~4, 793--812. \MR{2770907}

\bibitem{nourdin-peccati-reveillac:2010:multivariate-normal-approximation}
I. Nourdin, G. Peccati, and A. R{\'e}veillac, \emph{Multivariate
  normal approximation using {Stein}'s method and {Malliavin} calculus},
  Ann. Inst. Henri Poincar\'{e} Probab. Stat. \textbf{46} (2010),
  no.~1, 45--58. \MR{2641769}

\bibitem{nourdin-peccati-yang:2019:berry-esseen-bounds-breuer-major}
I. Nourdin, G. Peccati, and X. Yang, \emph{Berry-{E}sseen bounds
  in the {B}reuer-{M}ajor {CLT} and {G}ebelein's inequality}, Electron. Commun.
  Probab. \textbf{24} (2019), Paper No. 34, 12. \MR{3978683}

\bibitem{nourdin-rosinski:2014:asymptotic-independence-multiple}
I. Nourdin and J. Rosi{\'n}ski, \emph{Asymptotic independence of multiple
  {Wiener}-{It{\^o}} integrals and the resulting limit laws},
Ann. Probab. \textbf{42} (2014), no.~2, 497--526. \MR{3178465}

\bibitem{nualart-ortiz-latorre:2008:central-limit-theorems}
D.~Nualart and S.~Ortiz-Latorre, \emph{Central limit theorems for multiple
  stochastic integrals and {Malliavin} calculus}, Stochastic
Process. Appl. \textbf{118} (2008), no.~4, 614--628. \MR{2394845}

\bibitem{nualart-zhou:2018:total-variation-estimates}
D.~{Nualart} and H.~{Zhou}, \emph{Total variation estimates in the
  {B}reuer-{M}ajor theorem}, \ARXIV{1907.05230}

\bibitem{nualart:2006:malliavin-calculus-related}
D. Nualart, \emph{The {Malliavin} calculus and related topics}, second ed.,
  Probability and its {Applications} ({New} {York}), Springer-Verlag, Berlin,
  2006. \MR{2200233}

\bibitem{nualart-peccati:2005:central-limit-theorems}
D. Nualart and G. Peccati, \emph{Central limit theorems for sequences
  of multiple stochastic integrals}, Ann. Probab. \textbf{33}
  (2005), no.~1, 177--193. \MR{2118863}

\bibitem{samko-kilbas-marichev:1993:fractional-integrals-derivatives}
S.~G. Samko, A.~A. Kilbas, and O.~I. Marichev, \emph{Fractional
  integrals and derivatives}, Gordon and Breach Science Publishers, Yverdon,
  1993, Theory and applications, Edited and with a foreword by S. M.
  Nikolski\u{\i}, Translated from the 1987 Russian original, Revised by the
  authors. \MR{1347689}

\bibitem{shih:2011:steins-method-infinite-dimensional}
H.~H. Shih, \emph{On {Stein}'s method for infinite-dimensional {Gaussian}
  approximation in abstract {Wiener} spaces}, J. Funct. Anal.
  \textbf{261} (2011), no.~5, 1236--1283. \MR{2807099}

\end{thebibliography}

\end{document}